\documentclass[11pt, a4paper, english]{amsart}
\usepackage[top=2cm, bottom=2cm, left=2.5cm, right=2.5cm, twoside=false]{geometry}
 \usepackage[english]{babel} 
 \usepackage[utf8]{inputenc}
\usepackage{color}
\usepackage{amstext}
\usepackage{amsthm}
\usepackage{amssymb}
\usepackage{amscd}
\usepackage{tikz-cd}
\usepackage{booktabs}
\usepackage{tabularx}

\usepackage{graphicx}
\usepackage{adjustbox}

\usepackage{bbm}
\usepackage{comment}
 
\numberwithin{equation}{section}
\numberwithin{figure}{section}
\theoremstyle{plain}
\newtheorem*{thm*}{\protect\theoremname}
\newtheorem{thm}{\protect\theoremname}
\newtheorem{lem}[thm]{\protect\lemmaname}
\newtheorem{cor}[thm]{\protect\corollaryname}
\newtheorem{prop}[thm]{\protect\propositionname}

\theoremstyle{definition}

\theoremstyle{remark}
\newtheorem{rem}[thm]{\protect\remarkname}

\newcommand{\C}{\mathbb{C}}

\renewcommand{\to}{\rightarrow}

\renewcommand{\,}{\ }
\newcommand{\blem}{\begin{lem}}
\newcommand{\elem}{\end{lem}}
\newcommand{\bproof}{\begin{proof}}
\newcommand{\eproof}{\end{proof}}
\newcommand{\bprop}{\begin{prop}}
\newcommand{\eprop}{\end{prop}}

\newtheorem{problem}{Problem}

\usepackage{xypic}
\usepackage{mathrsfs, setspace, fancyhdr, times, bm}
\usepackage[colorlinks=true,pagebackref=true]{hyperref} 

\usepackage{babel}

\usepackage{amsmath} 
\usepackage{mathtools}

\providecommand{\definitionname}{Definition}
\providecommand{\examplename}{Example}
\providecommand{\lemmaname}{Lemma}
\providecommand{\corollaryname}{Corollary}
\providecommand{\propositionname}{Proposition}
\providecommand{\remarkname}{Remark}
\providecommand{\theoremname}{Theorem}

\begin{document}

\title[Threefold weighted Fano hypersurfaces]{Automorphism groups and cylindricity of weighted hypersurface Fano threefolds}


\author{In-Kyun Kim}
\address{June E Huh Center for Mathematical Challenges, Korea Institute for Advanced Study, 85, Hoegiro Dongdaemun-gu, Seoul 02455, Republic of Korea }
\email{soulcraw@kias.re.kr}

\author{Takashi Kishimoto}
\address{Department of Mathematics, Faculty of Science, Saitama University, Saitama 338-8570, Japan}
\email{kisimoto.takasi@gmail.com}

\author{Joonyeong Won}
\address{Department of Mathematics, Ewha Womans University, 52, Ewhayeodae-gil, Seodaemun-gu, Seoul, 03760, Republic of Korea} 
\email{leonwon@ewha.ac.kr}

\begin{abstract}
It is well known that there are totally 130 deformation families of quasi-smooth terminal weighted hypersurface Fano threefolds and all members belonging to 95 families of Fano indices one are birationally rigid. Among remaining $35$ families, $15$ families have the property that very general members are not stably rational, in particular, any of them has a finite automorphism group and is not cylindrical. In the present paper, we will observe the cylindricity and the full automorphism groups of every member in the remaining $20$ families. Moreover, we deal with the cylindricity of their forms.
\end{abstract}

\maketitle

\vspace{-5mm} 

\section{Introduction}\label{section1} 
Throughout the article, we denote by $k$ a fixed algebraically closed field of characteristic zero otherwise mentioned. 
\subsection{}\label{1-1} 
The minimal model program (MMP) allows a given normal projective variety $X$ with ${\mathbbm Q}$-factorial, terminal singularities to be linked to a suitable normal projective varieties with singularities of the same type equipped with distinguished properties by means of a composite of finitely many divisorial contractions and small maps, so-called flips. In any dimension, if the canonical divisor $K_X$ is not pseudo-effective, then there exists a directed MMP $X \dasharrow X'$ with $X'$ being a total space equipped with a structure of Mori fiber space (Mfs) (cf. \cite{BCHM}). As Mfs of special type, Fano varieties (of rank one) receive a lot of attention by several reasons, for instance, from the viewpoint of $K$-stability, cylindricity and the automorphism groups (cf. \cite{Calabi}, \cite{KKW}, \cite{CPS19}). For example, the cylindricity of varieties plays an important role in a connection with unipotent geometry, flexibility of certain affine algebraic varieties (cf. \cite{KPZ1}, \cite{KPZ2}, \cite{KPZ3}), on the other hand, in case of $k={\mathbbm C}$ the existence of K\"ahler-Einstein metrics on smooth Fano varieties is equivalent to the $K$-polystability of them (cf. \cite{Calabi}, \cite{Xu-Book}). It is worthwhile to note that if the total space $X'$ of Mfs, which is an outcome of MMP beginning with a normal projective variety $X$ with ${\mathbbm Q}$-factorial, terminal singularities, is cylindrical, then so is $X$ (see Lemma \ref{lem:mmp}, see also \cite[Lemma 1.1]{DK6}). Thus it is an important problem to observe the cylindricity of rank one Fano varieties with ${\mathbbm Q}$-factorial, terminal singularities.  

\subsection{}\label{1-2} 
This paper will focus on Fano threefolds. If $X$ is a smooth rank one Fano threefold, then we know completely the answer concerning the structure of the connected component ${\rm Aut}^0 (X)$ of the automorphism group of $X$ (see \cite{CPS19}), nevertheless the complete description of the full automorphism group ${\rm Aut}(X)$ remains unknown. Meanwhile, still now we do not know completely the answer about the cylindricity of Fano threefolds, e.g., certainly we know that if a smooth prime Fano threefold $X$ of genus $9$ or $10$ belongs to a certain subset of codimension one in the corresponding moduli, then $X$ is cylindrical (see \cite{KPZ2}), but for those out of the subset of codimension one, the cylindricity is still unknown. Moreover, we do not know the cylindricity of any smooth prime Fano threefold of genus $7$. In what follows, instead of smooth Fano threefolds, our main interest lies in quasi-smooth weighted hypersurface Fano threefolds:
\[
X_d =\big{\{} \, F =0 \, \big{\}} \, \subseteq \, {\mathbbm P} \big{(} {a_0}_{x}, {a_1}_{y}, {a_2}_{z}, {a_3}_{t}, {a_4}_{w} \big{)} = {\rm Proj} \big{(} k \big{[} x,y,z,t,w \big{]} \big{)},
\]
where $F$ is a quasi-homogeneous polynomial of degree $d$; see \cite{Fl00} for the definition of quasi-smoothness. All the possibilities about the quintuple:
\[
\big{(} a_0, a_1, a_2, a_3, a_4; d \big{)} 
\]
are nowadays completely known, more precisely there are totally $130$ cases (cf. \cite{ABR}, \cite{BS}). Let us put:
\[
\iota_X := \sum_{j=0}^4 a_j -d \in \{ 1,2,3,4,5,6,7,8,9,11,13 \} 
\]
which is called the {\it Fano index} of $X$, in other words, the anti-canonical divisor $-K_X$ of $X$ satisfies $-K_X \sim_{\mathbbm Q} \iota_X A$ with maximum $\iota_X \in {\mathbbm Z}_{>0}$ for some Weil divisor $A$ on $X$. Note that the range of the Fano index as above is a posteriori known by classification (cf. \cite{ABR}, \cite{BS}). The birational property varies drastically depending on the Fano index $\iota_X$. In fact, there are $95$ families among $130$ ones in which $\iota_X$ is equal to $1$, and in that case {\it all} members belonging to these $95$ families are birationally rigid (cf. \cite{CP16}), in particular $X$ is never cylindrical. 
On the other hand, among the remaining $130-95=35$ families, {\it very general} members belonging to $15$ families are known to be irrational, more precisely known not to be stably rational by Okada \cite{ok19}, especially any very general members from these $15$ families are not cylindrical, moreover the results from \cite{Esser} allow to conclude that these have finite automorphism groups, in particular, reductive groups. {\color{black}{Among these 15 families, {\it all} members of No. 96, 97 ,98 are classically known to be irrational which are smooth cubic threefold, quartic double solid and double cone of a Veronese surface, respectively. And {\it every} member in five families No. 100, 101, 102, 103, 110 is birationally solid by Okada \cite{ok23} so that all members of the families are irrational. Recently, it is proved that {\it every} member in two families of No. 107, 116 is also irrational by Prokhorov \cite{pr24}, \cite{pr25}.  Irrationality of all members of remaining five families, No. 99, 108, 109, 117, 122 are still open. }}  

From now on, we are interested in the remaining $130-(95+15)=20$ families, namely members belonging to the families:
$$
 (\spadesuit) \, {\rm No.} 104, 105, 106, 111, 112, 113, 114, 115, 118, 119, 120, 121, 123, 124, 125, 126, 127, 128, 129, 130, 
$$
found in the list of \cite{ACP}. Note that any member belonging to these families $(\spadesuit )$ is relatively easily checked to be rational. However, the rationality does not necessarily imply the cylindricity and the infiniteness of the corresponding automorphism group. Indeed, even if we consider smooth rank one rational Fano threefolds, the ones with infinite automorphism groups are quite in minority (cf. \cite{CPS19}). The main purpose of this article lies in the complete understanding of the cylindricity and the description of the automorphism groups of {\it all} members in the families $(\spadesuit )$. 
At first, the following is a result on the cylindricity and the description of the connected component of quasi-smooth weighted hypersurface Fano threefolds in the families $(\spadesuit )$: 
\begin{thm}\label{thm:main} 
Let $X$ be any quasi-smooth terminal weighted hypersurface Fano threefold in the families $(\spadesuit)$. Then $X$ is cylindrical, more precisely $X$ contains an open subset isomorphic to ${\mathbbm A}_k^2 \times {({\mathbbm A}_k^1 \backslash \{ o \} )}$. 
Moreover, the information concerning the possibility to contain the affine $3$-space ${\mathbbm A}_k^3$, and the connected component ${\rm Aut}^0 (X)$ of the automorphism group is summarized in Table 1, in which {\rm Yes} (resp. {\rm No}) of the fourth column means that any member belonging to the corresponding family contains (resp. does not contain) ${\mathbbm A}_k^3$. 
\end{thm} 
\begin{table}[htbp]
    \centering
    \caption{}
    \label{tab:hogehoge}
    \begin{tabular}{@{}clccl@{} }
        \toprule
        No. & $X_d \subseteq {\mathbbm P}({a_0},{a_1},{a_2}, {a_3}, {a_4})$ & $\iota_X$ & ${\mathbbm A}^3$ &  ${\rm Aut}^0(X)$  \\  \midrule
        104 &  $X_2 \subseteq {\mathbbm P}(1,1,1,1,1)$ &  3 & Yes &  ${\rm PSO}_5$   \\   \hline
        105 &  $X_3 \subseteq {\mathbbm P}(1,1,1,1,2)$ &  3 & Yes &  ${\mathbbm G}_a^3 \rtimes {\mathbbm G}_{m}$   \\   \hline
        106 &  $X_4 \subseteq {\mathbbm P}(1,1,1,2,2)$ & 3  & No &  ${\mathbbm G}_m$   \\   \hline
        111 &  $X_4 \subseteq {\mathbbm P}(1,1,1,2,3)$ & 4  & Yes &  ${\mathbbm G}_a^5 \rtimes {\mathbbm G}_{m}$   \\   \hline
        112 &  $X_6 \subseteq {\mathbbm P}(1,1,2,3,3)$ & 4  & No &  $ {\mathbbm G}_{m}$  \\   \hline
        113 &  $X_4 \subseteq {\mathbbm P}(1,1,2,2,3)$ & 5  & Yes &  
        {\color{black}{ ${\mathbbm G}_a^5 \rtimes {\mathbbm G}_m^2$  }} \\   \hline
        114 &  $X_6 \subseteq {\mathbbm P}(1,1,2,3,4)$ & 5  & No &  $ {\mathbbm G}_a^2 \rtimes {\mathbbm G}_{m}  $   \\   \hline
        115 &  $X_6 \subseteq {\mathbbm P}(1,2,2,3,3)$ & 5  & No &  $  {\mathbbm G}_{m}  $   \\   \hline
        
        118 &  $X_6 \subseteq {\mathbbm P}(1,1,2,3,5)$ & 6  & Yes &  $ {\mathbbm G}_a^7 \rtimes {\mathbbm G}_{m}$  \\   \hline
        119 &  $X_6 \subseteq {\mathbbm P}(1,2,2,3,5)$ &  7  & Yes &  $\color{black}{\mathbbm{G}_a^5\rtimes \mathbbm{G}_m}$   \\   \hline
        120 &  $X_6 \subseteq {\mathbbm P}(1,2,3,3,4)$ &  7  & No &  $\color{black}{ {\mathbbm G}_a^2\rtimes \mathbbm{G}_m^2}$ \\   \hline
        121 &  $X_8 \subseteq {\mathbbm P}(1,2,3,4,5)$ &  7   & No &  $\color{black}{\mathbbm{G}_a \rtimes \mathbbm{G}_m}$  \\   \hline
        
        123 &  $X_6 \subseteq {\mathbbm P}(1,2,3,3,5)$ &  8  &Yes &  $ {\mathbbm G}_a^5 \rtimes {\mathbbm G}_{m}^2  $  \\   \hline
        124 &  $X_{10} \subseteq {\mathbbm P}(1,2,3,5,7)$ &  8  & No &  $  {\mathbbm G}_a^2 \rtimes {\mathbbm G}_{m}  $  \\   \hline
        
        125 &  $X_{12} \subseteq {\mathbbm P}(1,3,4,5,7)$ &  8  & No &  $ {\mathbbm G}_{m}$  \\   \hline

        126 &  $X_6 \subseteq {\mathbbm P}(1,2,3,4,5)$ & 9  & Yes &  $ {\color{black}{{\mathbbm G}_a^6 \rtimes {\mathbbm G}_{m}^2} }$  \\   \hline
        127 &  $X_{12} \subseteq {\mathbbm P}(2,3,4,5,7)$ &  9  & No &   $ {\mathbbm G}_{m}$  \\   \hline
        
        128 &  $X_{12} \subseteq {\mathbbm P}(1,4,5,6,7)$ & 11 & No &  $ {\mathbbm G}_a \rtimes {\mathbbm G}_{m}  $  \\   \hline
        129 &  $X_{10} \subseteq {\mathbbm P}(2,3,4,5,7)$ & 11  & No &  $\color{black}{\mathbbm{G}_a \rtimes \mathbbm{G}_m}$ \\ \hline
        
        130 &  $X_{12} \subseteq {\mathbbm P}(3,4,5,6,7)$ & 13  & No &  $ {\mathbbm G}_{m}$  \\    \bottomrule
                
    \end{tabular}
\end{table}
Furthermore we will describe the full automorphism groups of all quasi-smooth weighted hypersurface Fano threefolds in the families $(\spadesuit )$. Determining the full automorphism groups is actually important for several reasons, e.g., to understand the forms of corresponding Fano varieties, which are not necessarily defined over an algebraically closed field (cf. \cite{Po}). Several families in $(\spadesuit )$ have non-trivial moduli, hence the description of ${\rm Aut}/ {\rm Aut}^0$ is not necessarily uniquely determined for such families.  
\begin{thm}\label{thm:full-auto} 
Let $X$ be a quasi-smooth terminal weighted hypersurface Fano threefold in the families $(\spadesuit)$. Then the automorphism group ${\rm Aut}(X)$ is isomorphic to ${\rm Aut}^0(X) \rtimes {\mathbb W}$, where ${\rm Aut}^0(X)$ is described in Table 1 and all possibilities of finite groups ${\mathbb W}={\rm Aut}(X)/{\rm Aut}^0(X)$ are summarized in Table 2.
\end{thm} 


\begin{table}[htbp]
    \centering
    \caption{}
    \label{tab:hogehoge2}
    \begin{tabularx}{0.8\textwidth}{@{}cX@{} }
        \toprule
        No. &  ${\mathbbm W}={\rm Aut}(X) / {\rm Aut}^0 (X)$ \\ \midrule
        104 &  $\{1\}$ \\   \hline
        105 &  $\mathrm{Aut}({\mathbbm P}_k^2;\text{smooth plane cubic curve})$  \\   \hline
        106 &  $\mathrm{Aut}(\text{smooth plane quartic curve}) \rtimes \mathbb{Z}_2$  \\   \hline
        111 & $\mathbb{Z}_2^3$ \\   \hline
        112 & ${\mathbbm Z}_2$, \, ${\mathbbm Z}_2^2$, \, ${\mathbbm Z}_2 \times {\mathbbm Z}_4$, \, ${\mathbbm Z}_2^3$, \, ${\mathbbm Z}_2 \times {\mathbbm Z}_8$, \, ${\mathbbm Z}_2^4$, \, ${\mathbbm Z}_2^2 \times {\mathbbm Z}_4$, \, ${\mathbbm Z}_{2}^2 \times {\mathbbm Z}_5$, \, ${\mathbbm Z}_2 \times {\mathbbm A}_4$, \, ${\mathbbm Z}_2^2 \times {\mathbbm Z}_8$, \, ${\mathbbm Z}_2\times {\mathbbm Z}_4 \times {\mathbbm Z}_5$, \,${\mathbbm Z}_2^2 \times {\mathbbm Z}_3 \times {\mathbbm Z}_4$, \, ${\mathbbm Z}_2^4 \times {\mathbbm Z}_3$, \, ${\mathbbm Z}_2 \times {\mathbbm Z}_3 \times S$
        \\   \hline
        113 & $\mathbbm{Z}_2\times\mathbbm{Z}_4$  \\   \hline
        114 & ${\mathbbm Z}_2 \times S$ \\   \hline
        115 & $\mathbbm{S}_3\times \mathbbm{Z}_2$, \, $\mathbbm{Z}_2\times \mathbbm{Z}_2$, \, $\mathbbm{Z}_2$   \\   \hline
        118  & $\mathbbm{Z}_2^2\times \mathbbm{Z}_3$, \, $\mathbbm{Z}_2 \times {\mathbbm Z}_3 \times {\mathbbm Z}_4$, \, $\mathbbm{Z}_2^2 \times \mathbbm{Z}_3^2$ \\   \hline
        119 & $\mathbbm{S}_3\times \mathbbm{Z}_2$   \\   \hline
        120 & $\mathbbm{Z}_2^2 \times\mathbbm{Z}_3$  \\   \hline
        121 & ${\mathbbm Z}_8 \times {\mathbbm Z}_2, \, {\mathbbm Z}_8 \times {\mathbbm Z}_2^2, \, \mathbbm{Z}_8 \times \mathbbm{Z}_3 \times \mathbbm{Z}_2, \, {\mathbbm Z}_8 \times {\mathbbm Z}_4 \times {\mathbbm Z}_2$   \\   \hline
        123 & $\mathbbm{Z}_2\times \mathbbm{Z}_3$   \\   \hline
        124 & ${\mathbbm Z}_2^2 \times {\mathbbm Z}_5^2$, \, ${\mathbbm Z}_2 \times {\mathbbm Z}_5 \times {\mathbbm Z}_8$, \, ${\mathbbm Z}_2 \times {\mathbbm Z}_4 \times {\mathbbm Z}_5$, \, ${\mathbbm Z}_2^2 \times {\mathbbm Z}_5$ \,     \\   \hline
        
        125 & $ \mathbbm{Z}_{3} \times \mathbbm{Z}_{4}$, \, $\mathbbm{Z}_{3} \times \mathbbm{Z}_{4}^2$, \, $\mathbbm{Z}_{3}^2 \times \mathbbm{Z}_{4}$, \, $\mathbbm{Z}_{3}^2 \times {\mathbbm Z}_4^2$  \\   \hline

        126 & $\mathbbm{Z}_2$  \\   \hline
        127 &  $\mathbbm{Z}_2\times \mathbbm{Z}_3 \times {\mathbbm Z}_4$, \, $\mathbbm{Z}_{3}\times \mathbbm{Z}_4^2$, \, $\mathbbm{Z}_2\times \mathbbm{Z}_3^2\times \mathbbm{Z}_4$ \\   \hline
        128  & $\mathbbm{Z}_{2} \times \mathbbm{Z}_3 \times {\mathbbm Z}_4$, \, $\mathbbm{Z}_{2}\times \mathbbm{Z}_3 \times {\mathbbm Z}_8$, \, $\mathbbm{Z}_2 \times \mathbbm{Z}_{3}^2\times \mathbbm{Z}_4$\\ \hline
        129  & $ \mathbbm{Z}_2^2 \times \mathbbm{Z}_5$  \\ \hline
        130 & $\mathbbm{Z}_2 \times \mathbbm{Z}_3\times \mathbbm{Z}_4$ \\   
        \bottomrule
        \multicolumn{2}{l}{$\{1\}$ denotes the trivial group.}\\
        \multicolumn{2}{l}{$\mathbbm{D}_n$ denotes the dihedral group of order $n$.}\\
        \multicolumn{2}{l}{$\mathbbm{S}_4$ denotes the symmetric group of degree $4$.}\\
        \multicolumn{2}{l}{$\mathbbm{A}_4$ denotes the alternating group of degree $4$.}\\
        \multicolumn{2}{p{0.75\textwidth}}{$S$ denotes the automorphism group of the sextic form in two variables. In our context, it is one of the following: $\{1\}$, $\mathbbm{Z}_2$, $\mathbbm{Z}_5$, $\mathbbm{D}_4$, $\mathbbm{D}_6$, $\mathbbm{D}_{12}$ or $\mathbbm{S}_4$.}
    \end{tabularx}
\end{table}

The following result is a direct consequence of Theorem \ref{thm:main}, and it is a counterpart of \cite[Theorem 1.1, Corollary 1.2]{PS21}:
\begin{cor}\label{cor:reductive}
The automorphism group of any quasi-smooth weighted hypersurface Fano threefold in the families No.105, 111, 113, 114, 118, 119, 120, 121, 123, 124, 126, 128 and 129 is not reductive. On the other hand, the automorphism group of any quasi-smooth weighted hypersurface Fano threefold in the families No.104, 106, 112, 115, 125, 127 and 130 is reductive. 
\end{cor}

\subsection{}\label{1-3}
Letting ${\mathbbm U}$ be a linear algebraic unipotent group, we say that a normal variety $X$ {\it admits a ${\mathbbm U}$-structure} if $X$ is endowed with a ${\mathbbm U}$-action with the open orbit. In case of dimension $n$ less than or equal to two, a linear algebraic unipotent group ${\mathbbm U}$ of dimension $n$ is unique up to isomorphism, i.e., a vector group ${\mathbbm U} \cong {\mathbbm G}_a^n$. Whereas in case of $n \geqq 3$, ${\mathbbm U}$ is not necessarily a vector group. For example, there are two unipotent linear algebraic groups of dimension $n=3$ up to isomorphism: the vector group ${\mathbbm G}_a^3$ and the non-commutative one, so-called the Heisenberg group ${\mathbbm H}$, which is isomorphic to the sub-group of ${\rm GL}_3(k)$ consisting of upper-triangular unipotent matrices. Recall that smooth (hence terminal) del Pezzo surfaces $S$ admitting ${\mathbbm G}_a^2$-structure are quite restrictive, in fact, there are only four possibilities ${\mathbbm P}_k^2$, ${\mathbbm P}_k^1 \times {\mathbbm P}_k^1$, ${\mathbbm F}_1$ and the smooth del Pezzo surface $S_7$ of degree $7$. In particular, the condition that a smooth del Pezzo surface $S$ contains the affine plane ${\mathbbm A}_k^2$ is very far from the situation that $S$ admits a ${\mathbbm G}_a^2$-structure. The similar phenomenon takes place when we look into smooth rank one Fano threefolds. Indeed, smooth rank one Fano threefolds $X$ containing the affine $3$-space ${\mathbbm A}_k^3$ are completely understood by Furushima \cite{Fu93} and Prokhorov \cite{Pr92}, namely, $X$ is isomorphic to ${\mathbbm P}_k^3$, the smooth quadric hypersurface $Q$, the smooth quintic del Pezzo threefold $V_5$ and smooth prime Fano threefolds of degree $22$ belonging to a $4$-dimensional subset in the moduli. But only ${\mathbbm P}_k^3$ and $Q$ among them can admit a unipotent group structure, to be more precise, ${\mathbbm P}_k^3$ and $Q$ can be endowed with ${\mathbbm G}_a^3$-structure and ${\mathbbm H}$-structure at the same time. Note that smooth Fano threefolds admitting ${\mathbbm G}_a^3$-structures (resp. ${\mathbbm H}$-structures) are classified in \cite{HM20} (resp. \cite{DFKM}). As a result, almost all smooth Fano threefolds equipped with ${\mathbbm G}_a^3$-structures admit simultaneously ${\mathbbm H}$-structures. Whereas once we observe the unipotent group structures on weighted hypersurface Fano threefolds, we have the following as a corollary of the argument to obtain Theorem \ref{thm:main} and Theorem \ref{thm:full-auto}: 
\begin{cor}\label{cor:aut}
A quasi-smooth terminal weighted hypersurface Fano threefold $X$ admits a ${\mathbbm G}_a^3$-structure if and only if $X$ contains the affine $3$-space ${\mathbbm A}_k^3$, more precisely, if and only if $X$ is (any) member of the following $8$ families:
\[
{\rm No.} 104, 105, 111, 113, 118, 119, 123, 126. 
\]
Moreover, $X$ admits a ${\mathbbm H}$-structure if and only if $X$ is in the family ${\rm No.} 104$, namely, if and only if $X$ is isomorphic to the smooth quadric hypersurface $Q$.    
\end{cor}
\begin{proof}
The first assertion is verified by our argument in the section 3 to prove Theorem \ref{thm:main} and Theorem \ref{thm:full-auto}. As for the second assertion, assume that a quasi-smooth weighted hypersurface Fano threefold $X$ belonging to the families $(\spadesuit )$ other than No.104 admits a ${\mathbbm H}$-structure. Then $X$ contains the affine $3$-space ${\mathbbm A}_k^3$, so that we have ${\rm Aut}^0(X) \cong {\mathbbm G}_a^N \rtimes {\mathbbm G}_m^M$ with $N \geqq 3$, $M >0$ by Theorem \ref{thm:main}. The inclusion ${\mathbbm H} \hookrightarrow {\mathbbm G}_a^N \rtimes {\mathbbm G}_m^M$ gives rise to an inclusion ${\mathbbm H}/ ({\mathbbm H} \cap {\mathbbm G}_a^N) \hookrightarrow {\mathbbm G}_m^M$, which implies that ${\mathbbm H}/ ({\mathbbm H} \cap {\mathbbm G}_a^N)$ is trivial, in other words, ${\mathbbm H}$ contains a sub-group isomorphic to ${\mathbbm G}_a^N$. This is a contradiction.
\end{proof}



\subsection{}\label{1-3.5}
The cylindricity of rank one Fano varieties defined over a field of characteristic zero $\Bbbk$, which is not necessarily algebraically closed, is often important to observe the cylindricity of the total spaces of given Mori fiber spaces via the concept of vertical cylinders, especially in case that $\Bbbk$ is the function field of the base variety (cf. \cite{DK4}, \cite{DK5}). Let $k:= \overline{\Bbbk}$ denote the algebraic closure of $\Bbbk$. Let $Y$ be a rank one Fano variety defined over $\Bbbk$ such that its base extension $Y_k:= Y \times_{{\rm Spec} \, \Bbbk } {{\rm Spec} \, k}$ contains the affine space ${\mathbbm A}_k^n$ with $n=\dim (Y_k)=\dim (Y)$. Even in such a situation, the initial Fano variety $Y$ does not always contain ${\mathbbm A}_{\Bbbk}^n$, e.g., pointless Severi-Brauer varieties, pointless quadric hypersurfaces, quintic del Pezzo threefold without special lines defined over $\Bbbk$ do not contain the affine space of the same dimension though their base extensions to the algebraic closure of $\Bbbk$ do (cf. \cite{DK5}). Nevertheless, once we observe $\Bbbk$-forms of quasi-smooth, terminal weighted Fano threefold hypersurfaces containing the affine $3$-space other than the smooth quadric hypersurface, i.e., the ones belonging to the families No.105, 111, 113, 118, 119, 123, 126 (see Theorem \ref{thm:main}), we have the following result:
\begin{thm}\label{thm:form}
Let $\Bbbk$ be a field of characteristic zero and let $Y$ be a variety defined over $\Bbbk$ such that its base extension to the algebraic closure is isomorphic to a quasi-smooth weighted Fano threefold hypersurface in the families No.105, 111, 113, 118, 119, 123, 126. Then $Y$ contains the affine $3$-space defined over $\Bbbk$.    
\end{thm}
\begin{rem}\label{rem:form1}
Note that we need to get rid off No.104 in the assertion of Theorem \ref{thm:form}. For example, a pointless ${\mathbbm R}$-form defined by:
\[
x^2+y^2+z^2+t^2+w^2=0 \, \subseteq \, {\rm Proj} \big{(} {\mathbbm R} \big{[} x,y,z,t,w \big{]} \big{)}
\]
does not contain ${\mathbbm A}_{{\mathbbm R}}^3$.  
\end{rem}
\begin{rem}\label{rem:form2}
In consideration of Theorem \ref{thm:full-auto}, the isomorphic classes of $\Bbbk$-forms of a given weighted Fano threefold hypersurface in the families No.105, 111, 113, 118, 119, 123, 126 are not necessarily uniquely determined. However Theorem \ref{thm:form} says that {\it all} of their $\Bbbk$-forms contain ${\mathbbm A}_{\Bbbk}^3$. The special kind of smooth prime Fano threefold $X_{22}^a$ of degree $22$ with ${\rm Aut}^0 (X_{22}^a) \cong {\mathbbm G}_a$ satisfies the same property (cf. \cite{DFK}), but to our knowledge $X_{22}^a$ was so far the only known example of Fano threefolds with this property other than Fano threefolds belonging to No.105, 111, 113, 118, 119, 123 and 126.
\end{rem}
The following is the immediate consequence of Theorem \ref{thm:form} in consideration of vertical cylinders, see \cite{DK4}, \cite{DK5}:
\begin{cor}\label{cor:form}
Let $\pi: Y \to Z$ be a Mori fiber space such that the general closed fibers of $\pi$ are isomorphic to quasi-smooth weighted Fano threefold hypersurfaces in the families No.105, 111, 113, 118, 119, 123 and 126. Then $Y$ contains a vertical ${\mathbbm A}^3$-cylinder with respect to $\pi$, i.e., there exists an open subsets $U$ of $Y$ and $U_0$ of $Z$ such that $U \cong U_0 \times {\mathbbm A}_{k}^3$ and $\pi|_U$ coincides with the projection ${\rm pr}_{U_0}: U \cong U_0 \times {\mathbbm A}_{k}^3 \to U_0$.
\end{cor}
In consideration of Corollary \ref{cor:reductive} and Theorem \ref{thm:form}, the following question deserves a consideration:
\begin{problem}\label{prob}
Let $\Bbbk$ be a field of characteristic zero, and let $Y$ be an algebraic variety of dimension $n$ defined over $\Bbbk$. Assume that the base extension $Y_{\overline{\Bbbk}}$ is a Fano variety such that $Y_{\overline{\Bbbk}}$ contains the affine $n$-space over ${\overline{\Bbbk}}$ and ${\rm Aut} (Y_{\overline{\Bbbk}})$ is not reductive. Then $Y$ contains the affine $n$-space defined over $\Bbbk$ ?  \end{problem}
Note that Problem \ref{prob} has an affirmative answer as long as we stick to either smooth del Pezzo surfaces\footnote{We can see without difficulty that all $\Bbbk$-forms of the Hirzebruch surface ${\mathbbm F}_1$ of degree $1$ and the smooth del Pezzo surface $S_7$ of degree $7$ contain the affine plane ${\mathbbm A}_{\Bbbk}^2$ defined over $\Bbbk$. Note that ${\mathbbm F}_1$ and $S_7$ are the only smooth del Pezzo surfaces with non-reductive automorphism groups.} or smooth rank one Fano threefolds (cf. \cite{DFK}) or quasi-smooth weighted Fano threefold hypersurfaces by Theorem \ref{thm:main} and Theorem \ref{thm:form}. 
\subsection{}\label{1-4} 
The paper proceeds as in the following scheme: In the section 2, at first we will prepare some useful lemmas for our purpose to find cylinders in Fano varieties. Then we shall show that {\it all} members $X$ of quasi-smooth, terminal weighted hypersurfaces Fano threefolds in the families $(\spadesuit)$ contain always ${\mathbbm A}^2$-cylinders isomorphic to ${\mathbbm A}_k^2 \times {({\mathbbm A}_k^1 \backslash \{ o \} )}$. Then we will observe which families in $(\spadesuit)$ contain furthermore the affine $3$-space ${\mathbbm A}_k^3$. In the section 3, we will focus on the description of the full automorphism groups for the families in $(\spadesuit)$. The connected component ${\rm Aut}^0 (X)$ is not very complicated to describe in general, meanwhile some families in $(\spadesuit)$ have non-trivial moduli, hence the description of the discrete part ${\rm Aut}(X) / {\rm Aut}^0 (X)$ is more complicated for such families, namely we need to proceed by case-by-case argument depending on the defining equations. Especially, we need to bear a troublesome case-by-case observation for No.112 to obtain the classification on ${\rm Aut}(X) / {\rm Aut}^0 (X)$. In the section 4, we will give the proof for Theorem \ref{thm:form}. 

\vspace{2mm} 

\noindent {\bf Acknowledgements}: The authors thank SKYBAY Hotel Gyeongpo in Gangneung of Korea, in which this project was initiated when Conference: Affine Geometry and Birational Geometry was held in February 2025, and IBS POSTECH and Saitama University at which it was continued during a visit of three authors and a visit of the first and the third authors for their excellent working conditions offered. Moreover, the authors are grateful to Takuzo Okada for his useful comment on the references. The second author was partially supported by JSPS KAKENHI Grant Number 23K03047. The first and third authors was supported by the National Research Foundation of Korea [NRF-2023R1A2C1003390 to I.-K.K.; RS-2025-00513064 to J.W.].

\section{Cylinders in Fano varieties}\label{section2} 
In this section, we will begin with several results which are often useful to find cylinders contained in Fano varieties. Then we apply the results to show that all members contained in the families $(\spadesuit)$ contain ${\mathbbm A}^2$-cylinders. Moreover, we will observe which families contain the affine $3$-space ${\mathbbm A}_k^3$.

\subsection{Some Lemmas}\label{2-1} 
The following simple Lemma \ref{lem:mmp} is not used for later use. But it will yield to some extent a motivational reason why we focus on finding cylinders in Fano varieties. 
\begin{lem}\label{lem:mmp} 
Let $X$ be a normal projective variety with ${\mathbbm Q}$-factorial, terminal singularities whose canonical divisor $K_X$ is not pseudo-effective. Let $X'$ be an outcome of MMP $ X \dasharrow X'$ beginning with $X$. If $X'$ is cylindrical, then so is $X$.
\end{lem} 
\begin{proof} 
By induction on the number of divisorial contractions and flips appearing in the process of MMP $ \psi: X \dasharrow X'$, we have only to prove the assertion for the case where $\psi$ is either a single divisorial contraction or a flip. If $\psi$ is a divisorial contraction, then it is easy to see that the cylindricity of $X'$ implies that of $X$. On the other hand, if $\psi$ is a flip, then we need only to apply \cite[Lemma 9]{DK6}.
\end{proof} 
\begin{rem}\label{rem:cyl}
The stronger version of Lemma \ref{lem:mmp} holds true for smooth projective surfaces by \cite{Saw23}: Letting $S_1 \dasharrow S_2$ be a birational map between smooth projective surfaces, $S_1$ is cylindrical if and only if so is $S_2$. However, even in case of dimension $2$, once we observe the cylindricity outside the category of terminal singularities, the similar result does not hold in general. For example, let $S$ be a Du Val del Pezzo surface of degree $1$ with:
\[
{\rm Sing}(S) = 2 {\rm D}_4, \, 2{\rm A}_3+2{\rm A}_1, \, 4{\rm A}_2.
\]
Then it is known that $S$ does not contain any cylinder, see \cite{CPW16} or \cite{Saw24}. Let $\mu: X\to S$ be the minimal resolution. Since $X$ is rational, $X$ contains the affine plane, in particular, $X$ is cylindrical.
\end{rem}
For our purpose to find cylinders in given Fano varieties, the next simple result is often useful. 
\begin{lem}\label{lem:cylinder} 
Let ${\mathbbm A}_k^n (x_1, x_2, \cdots , x_n) / \! / {\bf \mu}_d (1,a_2, \cdots , a_n)$ denote the quotient of the affine $n$-space ${\mathbbm A}_k^n$ by the cyclic group ${\bf \mu}_n = \langle \zeta_d \rangle$ defined by:
\[
\zeta_d \cdot \big{(} x_1, x_2, \cdots , x_n \big{)}= \big{(} \zeta_d x_1, \zeta_d^{a_2} x_2, \cdots , \zeta_d^{a_n} x_n \big{)}, 
\]
where $\zeta_d$ is a primitive $d$-th root of unity. Then ${\mathbbm A}_k^n (x_1, x_2, \cdots , x_n) / \! / {\bf \mu}_d (1,a_2, \cdots , a_n)$ contains an ${\mathbbm A}^{n-1}$-cylinder isomorphic to ${\mathbbm A}_k^{n-1} \times ( {\mathbbm A}_k^1 \backslash \{ o \} )$. 
\end{lem} 
\begin{proof} 
The argument is similar to that in \cite{KPZ1} or \cite{KKW+}. Let $X:= {\mathbbm A}_k^n (x_1, x_2, \cdots , x_n) / \! / {\bf \mu}_d (1,a_2, \cdots , a_n)$ and $Y:={\mathbbm P}(1_{x_1}, {a_2}_{x_2}, \cdots, {a_n}_{x_n})$, and let:
\[
p: X\backslash \{ o \}  \, \longrightarrow \, Y 
\]
be the natural projection associated to the ${\mathbbm G}_m$-quotient. On the other hand, let:
\[
\mu : \widetilde{X} \longrightarrow X 
\]
be the $\frac{1}{d}(1, a_2, \cdots , a_n)$-blowing-up at the singularity $o \in X$ with the exceptional divisor $\tilde{E} \cong Y$, so that $\tilde{X}$ is equipped with a ${\mathbbm P}^1$-bundle structure:
\[
\pi: \widetilde{X} \longrightarrow Y. 
\]
Let us put $H:= {\mathbbm V}_+ (x_1) \subseteq Y$. 
Then all fibers of the restriction:
\[
{\pi}|_{\widetilde{X} \backslash \pi^\ast (H) }: \widetilde{X} \backslash \pi^\ast (H) \longrightarrow Y \backslash H \cong {\rm Spec} \Big{(} k \Big{[} \frac{x_2}{x_1^{a_2}}, \cdots, \frac{x_n}{x_1^{a_n}} \Big{]} \Big{)} \cong {\mathbbm A}_k^{n-1} 
\]  
are isomorphic to the affine line over the respective residue fields, hence ${\pi}|_{\widetilde{X} \backslash \pi^\ast (H) }$ is a trivial ${\mathbbm A}^1$-bundle:
\[
\widetilde{X} \backslash \pi^\ast (H) \cong (Y \backslash H) \times {\mathbbm A}_k^1 \cong 
 {\rm Spec} \Big{(} k \Big{[} \frac{x_2}{x_1^{a_2}}, \cdots, \frac{x_n}{x_1^{a_n}} \Big{]} \Big{)} \times {\rm Spec} \big{(} k [ t ] \big{)}
\]
Note that the restriction $\tilde{E} \cap ( \widetilde{X} \backslash \pi^\ast (H) )$ corresponds to $(Y \backslash H) \times \{ \ast \}$ via the above isomorphism. Therefore we have:
\[
X \backslash H \cong \widetilde{X} \backslash ( \pi^\ast (H) \cup \widetilde{E}) \cong ({\mathbbm A}_k^{n-1} \times {\mathbbm A}_k^1) \backslash ({\mathbbm A}_k^{n-1} \times \{ \ast \} ) \cong {\mathbbm A}_k^{n-1} \times {\rm Spec} \big{(} k [t^{\pm} ] \big{)}, 
\]
in particular, $X$ contains an ${\mathbbm A}^{n-1}$-cylinder isomorphic to ${\mathbbm A}_k^{n-1} \times ( {\mathbbm A}_k^1 \backslash \{ o \} )$ as desired. 
\end{proof} 
The following lemma, which is used to show that several Fano threefolds do not contain the affine $3$-space ${\mathbbm A}_k^3$, seems well known to affine algebraic geometers. 
\begin{lem}\label{lem:am} 
Let $U$ be the affine algebraic threefold defined by: 
\[
st + f(u,v) =0 
\]
in ${\mathbbm A}_k^4={\rm Spec}(k[s,t,u,v])$, where $f (u,v) \in k[u,v]$ is a non-constant polynomial in two variables $u,v$. Then $U$ is isomorphic to the affine $3$-space ${\mathbbm A}_k^3$ if and only if $f(u,v)=0$ defines scheme-theoretically the affine line in the affine plane ${\mathbbm A}_k^2={\rm Spec}(k[u,v])$.
\end{lem} 
\begin{proof} 
Assume that $f(u,v)=0$ defines scheme-theoretically the affine line in ${\mathbbm A}_k^2={\rm Spec}(k[u,v])$. Then a theorem of Abhyankar-Moh-Suzuki (see for instance \cite[Chapter 2]{Miy78}), we may assume from the beginning that $f(u,v)=u$, hence $U \cong {\mathbbm A}_k^3$ obviously. Conversely, suppose that $U$ is isomorphic to the affine $3$-space ${\mathbbm A}_k^3$. We look at the restriction of the projection to the $t$-axis ${\rm pr}_t :{\mathbb A}_k^4={\rm Spec}(k[u,v,s,t]) \to {\mathbbm A}_k^1={\rm Spec}(k [t])$ onto $U$, say:
\[
p:= {{\rm pr}_t}|_U: \, U \, \longrightarrow \, {\mathbb A}_k^1={\rm Spec}(k [t]). 
\]
Any closed fiber of $p$ other than that over the origin is isomorphic to ${\mathbbm A}_k^2$, which implies that all fibers of $p$ are isomorphic to ${\mathbb A}_k^2$ by \cite{Kaliman}, in particular $p^\ast (o)\cong {\mathbbm A}_k^2$. This implies in turn that the affine curve defined by $f(u,v)=0$ in ${\mathbbm A}_k^2={\rm Spec}(k[u,v])$ is isomorphic to the affine line. 
\end{proof} 
The next lemme is also useful to disprove the existence of the affine $3$-space ${\mathbbm A}_k^3$ contained in certain weighted hypersurface Fano threefolds in the families $(\spadesuit)$ combined with Lemma \ref{lem:am}. 
\begin{lem}\label{lem:cr} 
Let $X \subseteq {\mathbbm P}= {\mathbbm P} (a_0, a_1, \cdots, a_{n+1})$ be a quasi-smooth, well-formed weighted hypersurface of degree $d$, where weights satisfy $a_{n+1} \geqq \cdots \geqq a_1 \geqq a_0 >0$ with $n \geqq 3$. If $X$ contains an open subset isomorphic to the affine $n$-space ${\mathbbm A}_k^n$, then $a_0=1$ and the complement $H=X \backslash U$ is a restriction of a suitable hypersurface of degree $1$ in ${\mathbbm P}$.


\end{lem} 
\begin{proof} 
At first, we note that the class group ${\rm Cl}(X)$ of $X$ is free generated by the class of the reflexive sheaf ${\mathscr O}_X (1):={\mathscr O}_{{\mathbbm P}}(1)|_X$, see e.g., \cite[Theorem 2.15]{PS21}. Assume now that $X$ contains an open subset $U$, which is isomorphic to ${\mathbbm A}_k^n$. 
Since the coordinate ring $k[U]$ of $U$ is UFD and $X$ is of rank one, ${\rm Cl}(X)$ is free generated by the boundary divisor $H:= X \backslash U$. As $X$ is ${\mathbbm Q}$-factorial (because $X$ is assumed to be quasi-smooth, in particular $X$ has only quotient singularities), $H$ is actually irreducible. This implies in turn that the degree one piece:
\[
S_1 \cong H^0 \big{(} X, {\mathscr O}_X (1) \big{)}
\]
of the graded $k$-algebra:
\[
S= k \big{[} x_0, x_1, \cdots , x_{n+1} \big{]} \big{/} \big{\langle} f \big{\rangle} =\bigoplus_{\ell \geqq 0} S_\ell, 
\]
is nonzero. Therefore we have $a_0=1$. The remaining assertion is now easy to see. 
\end{proof} 
\begin{rem}\label{rem:cr}
It is important to assume that a quasi-smooth weighted hypersurface $X$ in a well-formed weighted projective space ${\mathbbm P}$ is well-formed. For instance, we consider:
\[
X=\{ x=0 \} \subseteq {\mathbbm P}(2_x, 3_y,3_z, 3_t, 5_w), 
\]
which is quasi-smooth but not well-formed. On the other hand:
\[
X \cong {\mathbbm P}(3_y,3_z,3_t,5_w) \cong {\mathbbm P}(1,1,1,5)
\]
contains the affine $3$-space.
\end{rem}
\subsection{Cylindricity of weighted hypersurface Fano threefolds from the families $(\spadesuit)$}\label{2-2} 
In this subsection, we will show that all members from the families $(\spadesuit)$ contain at least ${\mathbbm A}^2$-cylinder. More precisely, we will divide the families of $(\spadesuit)$ into two parts: families whose any member contains the affine $3$-space ${\mathbbm A}_k^3$, and those whose any member does not contain ${\mathbbm A}_k^3$ but contains an open subset isomorphic to $({\mathbbm A}_k^1 \backslash \{ 0 \}) \times {\mathbbm A}_k^2$. The argument lies in case by case analysis depending on the families in $(\spadesuit)$ combined with Lemmas in the previous subsection. 
\subsubsection{{\rm No. 104}}\label{104} 
Let $X$ be in the family No. 104. Then we may assume that:
\[
X_2: xy+z^2+t^2+w^2=0 \subseteq {\mathbbm P}(1_x,1_y,1_z,1_t,1_w). 
\]
Hence $X \backslash {\mathbbm V}_+ (x)$ is isomorphic to ${\mathbbm A}_k^3$. 

\subsubsection{{\rm No. 105}}\label{105} 
Let $X$ be in the family No. 105. Then we may assume that $X$ is defined by:
\[
X_3: tw+f(x,y,z) =0 \subseteq {\mathbbm P}(1_x,1_y,1_z,1_t,2_w), 
\]
where $f(x,y,z)$ is a cubic homogeneous polynomial. The quasi-smoothness condition implies that the plane curve $C:= {\mathbbm V}_+(f) \subseteq {\mathbbm P}(1_x,1_y,1_z)$ is a smooth elliptic curve. The affine chart $X \cap {\mathbbm D}_+ (t)$ is obviously isomorphic to ${\mathbbm A}_k^3$.

\subsubsection{{\rm No. 106}}\label{106} 
Let $X$ be in the family No. 106. Then we may assume that $X$ is defined by:
\[
X_4: tw+f(x,y,z) =0 \subseteq {\mathbbm P}(1_x,1_y,1_z,2_t,2_w), 
\]
where $f(x,y,z)$ is a quartic homogeneous polynomial. The quasi-smoothness of $X$ implies that the plane curve $C:= {\mathbbm V}_+(f) \subseteq {\mathbbm P}(1_x,1_y,1_z)$ is a smooth curve of genus $g(C)=3$. Notice that the affine chart $X \cap {\mathbbm D}_+ (w)$ is isomorphic to:
\[
{\mathbbm A}_k^3 (x,y,z) / \! / {\bf \mu}_2 (1,1,1), 
\]
hence $X \cap {\mathbbm D}_+ (w)$ contains an ${\mathbbm A}_k^2$-cylinder isomorphic to $({\mathbbm A}_k^1 \backslash \{ o \} ) \times {\mathbbm A}_k^2$ by Lemma \ref{lem:cylinder}. On the other hand, if $X$ contains an open subset $U$ isomorphic to ${\mathbbm A}_k^3$, then the complement $H=X \backslash U$ would be a restriction of a suitable hypersurface of degree $1$, see Lemma \ref{lem:cr}. Hence we may assume that $H={\mathbbm V}_+ (x) \cap X$ by means of change of coordinates. Then $U$ is isomorphic to the hypersurface in ${\mathbbm A}_k^4={\rm Spec}(k [y,z,t,w])$ defined by:
\[
tw+f(1,y,z) =0. 
\]
By Lemma \ref{lem:am}, the polynomial $f(1,y,z) \in k[y,z]$ should define the affine line in ${\mathbbm A}_k^2={\rm Spec}(k[y,z])$. This is absurd because $C$ is not rational. 

\subsubsection{{\rm No. 111}}\label{111} 
Let $X$ be in the family No. 111. Then we may assume that $X$ is defined by:
\[
X_4: zw+t^2 + f_4(x,y) =0 \subseteq {\mathbbm P}(1_x,1_y,1_z,2_t,3_w), 
\]
where $f_4(x.y)$ is a homogeneous polynomial of degree $4$ in $x,y$. 
It is then easy to find the affine $3$-space in $X$, indeed, the affine chart $X \cap {\mathbbm D}_+ (z)$ is isomorphic to ${\mathbbm A}_k^3$.

\subsubsection{{\rm No. 112}}\label{112} 
Let $X$ be in the family No. 112. Then we may assume that $X$ is defined by:
\[
X_6: tw+f_6(x,y,z) =0 \subseteq {\mathbbm P}(1_x,1_y,2_z,3_t,3_w), 
\]
where $f_6(x,y,z)$ is a quasi-homogeneous polynomial of degree $6$ with respect to the weight ${\rm wt}(x,y,z)=(1,1,2)$ such that the curve $C:= {\mathbbm V}_+ (f_6) \subseteq {\mathbbm P}(1_x,1_y,2_z)$ is quasi-smooth. The affine chart $X \cap {\mathbbm D}_+ (w)$ is isomorphic to:
\[
{\mathbbm A}_k^3 (x,y,z) / \! / {\bf \mu}_3 (1,1,2) 
\]
Hence $X$ contains the open subset isomorphic to ${\mathbbm A}_k^{2} \times ( {\mathbbm A}_k^1 \backslash \{ o \} )$ by Lemma \ref{lem:cylinder}. 
Now we write:
\[
f_6(x,y,z) = a z^3+ g_2(x,y)z^2+g_4(x,y)z+g_6(x,y), 
\]
where $a \in k$ and $g_i \in k[x,y]$ is a homogeneous polynomial of degree $i$ in $x, y$. The curve $C \subseteq {\mathbbm P}(1_x,1_y,2_z)$ being quasi-smooth, we have $a \neq 0$, so we may assume that:
\[
f_6(x,y,z)=z^3+ g_4(x,y) z + g_6(x,y), 
\]
which says that $C$ is smooth of genus $g(C)=4$. This allows to conclude that $X$ does not contain ${\mathbbm A}_k^3$. Indeed, assume on the contrary that $X$ contains an open subset $U$ isomorphic to ${\mathbbm A}_k^3$. Then we may assume that the boundary $H:=X \backslash U$ is the restriction of ${\mathbbm V}_+ (x)$ up to change of coordinates by Lemma \ref{lem:cr}, hence $U$ is isomorphic to the affine hypersurface in ${\mathbbm A}_k^4={\rm Spec}(k[y,z,t,w])$ defined by:
\[
tw+ f_6 (1,y,z) =0. 
\]
Lemma \ref{lem:am} then implies that $f_6(1,y,z)=0$ defines the affine line in ${\mathbbm A}_k^2={\rm Spec}(k[y,z])$. This is a contradiction because $C$ is not rational.


\subsubsection{{\rm No. 113}}\label{113} 
Let $X$ be in the family No. 113. Then in consideration of the quasi-smoothness of $X$, we may assume that $X$ is defined by:
\[
X_4: yw+t^2 + z^2 + x^4 =0 \subseteq {\mathbbm P}(1_x,1_y,2_z,2_t,3_w). 
\]
Then the affine chart $X \cap {\mathbbm D}_+ (y)$ is isomorphic to ${\mathbbm A}_k^3$. 

\subsubsection{{\rm No. 114}}\label{114} 
Let $X$ be in the family No. 114. As $X$ is assumed to be quasi-smooth, we may assume that $X$ is defined by:
\[
X_6: zw+t^2+f_6(x,y)  \subseteq {\mathbbm P}(1_x,1_y,2_z,3_t,4_w), 
\]
where $f_6(x,y)$ is a homogeneous polynomial of degree $6$ such that $f_6(x,y)=0$ yields distinct six points on ${\mathbbm P}_k^1={\rm Proj} (k [x,y])$. Then the curve $C:= {\mathbbm V}_+ (t^2+f_6) \subseteq {\mathbbm P}(1_x,1_y,3_t)$ is smooth of genus $g(C)=2$. Note that the affine chart $X \cap {\mathbbm D}_+ (w)$ is isomorphic to:
\[
{\mathbbm A}_k^3 (x,y,t) / \! / {\bf \mu}_4 (1,1,3), 
\]
which contains an ${\mathbbm A}^2$-cylinder isomorphic to ${\mathbbm A}_k^2 \times ({\mathbbm A}_k^1 \backslash \{ o \})$ by Lemma \ref{lem:cylinder}. Now supposing that $X$ contains an open subset $U$ isomorphic to ${\mathbbm A}_k^3$, Lemma \ref{lem:cr} implies that the boundary $H:=X \backslash U$ would be the restriction of ${\mathscr O}(1)$. Hence by a suitable change of coordinates, we may assume that $H={\mathbbm V}_+ (x) \cap X$, so that $U$ is isomorphic to the following affine hypersurface in ${\mathbbm A}_k^4={\rm Spec}(k [y,z,t,w])$:
\[
zw+t^2+f_6(1,y)=0, 
\]
which is never isomorphic to ${\mathbbm A}_k^3$ by virtue of Lemma \ref{lem:am} because $C$ is not rational. This is a contradiction. 

\subsubsection{{\rm No. 115}}\label{115} 
Let $X$ be in the family No. 115. Since $X$ is quasi-smooth, we may assume that $X$ is defined by:
\[
X_6: tw+f_3(y,z)+f_2(y,z)x^2+f_1(y,z)x^4+cx^6=0   \subseteq {\mathbbm P}(1_x,2_y,2_z,3_t,3_w), 
\]
where $f_i(y,z)$ is a homogeneous polynomial in $y,z$ of degree $i$ and $c \in k$. Furthermore, we may assume that $f_3(y,z)=y^3+z^3$ and at least one of $f_1(y,z)$ and $c$ is non-zero, in other words, the curve:
\[
C:={\mathbbm V}_+ (f_3(y,z)+f_2(y,z)x+f_1(y,z)x^2+cx^3) \subseteq {\mathbbm P}(1_x,1_y,1_z) 
\]
is a smooth elliptic curve. The affine chart $X \cap {\mathbbm D}_+ (w)$ is isomorphic to:
\[
{\mathbbm A}_k^3 (x,y,z) / \! / {\bf \mu}_3 (1,2,2), 
\]
which contains an ${\mathbbm A}^2$-cylinder isomorphic to ${\mathbbm A}_k^2 \times ({\mathbbm A}_k^1 \backslash \{ o \})$ by Lemma \ref{lem:cylinder}. Supposing that $X$ contains an open subset $U$ isomorphic to ${\mathbbm A}_k^3$, Lemma \ref{lem:cr} implies that the boundary $H:=X \backslash U$ would be the restriction of ${\mathscr O}(1)$. Hence we have $H={\mathbbm V}_+ (x) \cap X$, so that $U$ is isomorphic to the affine hypersurface in ${\mathbbm A}_k^4={\rm Spec}(k [y,z,t,w])$ defined by:
\[
tw+f_3(y,z)+f_2(y,z)+f_1(y,z)+ c=0, 
\]
which is not isomorphic to ${\mathbbm A}_k^3$ by virtue of Lemma \ref{lem:am} since $C$ is not rational. This contradiction means that $X$ never contains ${\mathbbm A}_k^3$. 

\subsubsection{{\rm No. 118}}\label{118} 
Let $X$ be in the family No. 118. Since $X$ is assumed to be quasi-smooth, we may assume that $X$ is defined as follows:
\[
X_6:  yw+t^2+z(z+x^2)(z+cx^2)=0  \subseteq {\mathbbm P}(1_x,1_y,2_z,3_t,5_w), 
\] 
with $c \in k \backslash \{ 0, 1 \}$. Then the affine chart $X \cap {\mathbbm D}_+ (y)$ is obviously isomorphic to ${\mathbbm A}_k^3$. 

\subsubsection{{\rm No. 119}}\label{119} 
Let $X$ be in the family No. 119. Since $X$ is quasi-smooth, we may assume that $X$ is defined as follows:
\[
X_6: xw+t^2+yz(y+z) =0 \subseteq {\mathbbm P}(1_x,2_y,2_z,3_t,5_w), 
\]
Then the affine chart $X \cap {\mathbbm D}_+ (x)$ is isomorphic to ${\mathbbm A}_k^3$.

\subsubsection{{\rm No. 120}}\label{120} 
Let $X$ be in the family No. 120. The quasi-smoothness of $X$ implies that $X$ is defined by:
\[
X_6: yw+zt+x^6 =0 \subseteq {\mathbbm P}(1_x,2_y,3_z,3_t,4_w). 
\]
Then the affine chart $X \cap {\mathbbm D}_+ (y)$ is isomorphic to:
\[
{\mathbbm A}_k^3 (x,z,t) / \! / {\bf \mu}_2 (1,1,1), 
\]
hence it contains an open subset isomorphic to ${\mathbbm A}_k^{2} \times ( {\mathbbm A}_k^1 \backslash \{ o \} )$ by Lemma \ref{lem:cylinder}. On the other hand, assuming that $X$ contains an open subset $U$ isomorphic to ${\mathbbm A}_k^3$, it must be that the boundary $H=X \backslash U$ coincides with the restriction ${\mathbbm V}_+ (x) \cap X$ by Lemma \ref{lem:cr}, thence $U$ is isomorphic to the hypersurface:
\[
yw+zt+1=0 
\]
in ${\mathbbm A}_k^4={\rm Spec}(k [y,z,t,w])$. This is a contradiction combined with Lemma \ref{lem:am} as the affine curve defined by $zt+1=0$ in ${\mathbbm A}_k^2={\rm Spec}(k[y,z])$ is not isomorphic to ${\mathbbm A}_k^1$. 

\subsubsection{{\rm No. 121}}\label{121} 
Let $X$ be in the family No. 121. Since $X$ is assumed to be quasi-smooth, $X$ is described as follows:
\[
X_8: zw+t^2+y(y+a_1x^2)(y+a_2x^2)(y+a_3x^2)=0 \subseteq {\mathbbm P}(1_x,2_y,3_z,4_t,5_w),  
\]
where $a_1, a_2, a_3 \in k \backslash \{ 0 \}$ are mutually distinct. Then the affine chart $X \cap {\mathbbm D}_+ (z)$ being isomorphic to:
\[
{\mathbbm A}_k^3 (x,y,t) / \! / {\bf \mu}_3 (1,2,1), 
\]
it contains an open subset isomorphic to ${\mathbbm A}_k^{2} \times ( {\mathbbm A}_k^1 \backslash \{ o \} )$ by Lemma \ref{lem:cylinder}. On the other hand, supposing that $X$ contains an open subset $U$ isomorphic to ${\mathbbm A}_k^3$, the boundary divisor $H=X \backslash U$ would coincide with the restriction ${\mathbbm V}_+ (x) \cap X$ by Lemma \ref{lem:cr}, thence $U$ is isomorphic to the hypersurface:
\[
zw+t^2+y(y+a_1)(y+a_2)(y+a_3)=0  
\]
in ${\mathbbm A}_k^4={\rm Spec}(k [y,z,t,w])$. But since the affine curve defined by:
\[
t^2+y(y+a_1)(y+a_2)(y+a_3)=0 
\]
in ${\mathbbm A}_k^2={\rm Spec}(k[y,t])$ is not isomorphic to ${\mathbbm A}_k^1$, this is a contradiction to Lemma \ref{lem:am}.

\subsubsection{{\rm No. 123}}\label{123} 
Let $X$ be in the family No. 123. Since $X$ is quasi-smooth, we may assume that $X$ is defined as follows:
\[
X_6: xw + zt + y^3 =0 \subseteq {\mathbbm P}(1_x,2_y,3_z,3_t,5_w). 
\]
Then the affine chart $X \cap {\mathbbm D}_+ (x)$ is isomorphic to ${\mathbbm A}_k^3$. 

\subsubsection{{\rm No. 124}}\label{124} 
Let $X$ be in the family No. 124. As $X$ is quasi-smooth, we may assume that $X$ is defined as follows:
\[
X_{10}: zw+t^2+y(y+a_1x^2)(y+a_2x^2)(y+a_3x^2)(y+a_4x^2)=0 \subseteq {\mathbbm P}(1_x,2_y,3_z,5_t,7_w), 
\]
where $a_1, a_2, a_3, a_4 \in k \backslash \{ 0 \}$ are mutually distinct constants. Then we can show that $X$ contains an open subset isomorphic to  ${\mathbbm A}_k^{2} \times ( {\mathbbm A}_k^1 \backslash \{ o \} )$, whereas $X$ does not contain ${\mathbbm A}_k^3$ by the almost same argument as in \ref{121}.  

\subsubsection{{\rm No. 125}}\label{125} 
Let $X$ be in the family No. 125. As $X$ is quasi-smooth, we may assume that $X$ is described as:
\[
X_{12} : tw+y^4+g_6(x,z)y^2+g_9(x,z)y+g_{12}(x,z) \subseteq {\mathbbm P}(1_x,3_y,4_z,5_t,7_w),  
\]
where $g_i(x,z)$ is a quasi-homogeneous polynomial of degree $i$ in $x,z$ with respect to the weight ${\rm wt}(x,z) = (1,4)$. Then by a further change of coordinates on $x$ and $z$ if necessary, we may assume that $X$ is defined by:
\[
X_{12} : tw+y^4+ z(z+x^4)(z-x^4)=0 \subseteq {\mathbbm P}(1_x,3_y,4_z,5_t,7_w)
\]
The affine chart $X \cap {\mathbbm D}_+ (w)$ being isomorphic to:
\[
{\mathbbm A}_k^3 (x,y,z) / \! / {\bf \mu}_7 (1,3,4), 
\]
it contains an open subset isomorphic to ${\mathbbm A}_k^{2} \times ( {\mathbbm A}_k^1 \backslash \{ o \} )$ by virtue of Lemma \ref{lem:cylinder}. On the other hand, suppose that $X$ contains an open subset $U$ isomorphic to ${\mathbbm A}_k^3$. Then the boundary divisor $H=X \backslash U$ would coincide with the restriction ${\mathbbm V}_+ (x) \cap X$ by Lemma \ref{lem:cr}. Thus $U$ is isomorphic to the hypersurface:
\[
tw+y^4+z(z+1)(z-1) =0 \, \subseteq \, {\mathbbm A}_k^4={\rm Spec}(k [y,z,t,w]). 
\]
Lemma \ref{lem:am} then says that the affine curve in ${\mathbbm A}_k^2={\rm Spec}(k [x,z])$ defined by:
\[
y^4+z(z+1)(z-1)=0
\]
should be isomorphic to ${\mathbbm A}_k^1$, which is obviously absurd for instance by the fact that the curve:
\[
C:={\mathbbm V}_+ \big{(} y^4+z(z+x^4)(z-x^4) \big{)} \subseteq {\mathbbm P} \big{(} 1_x, 3_y, 4_z \big{)}  
\]
is a smooth curve of genus $g(C)=3$. 

\subsubsection{{\rm No. 126}}\label{126} 
Let $X$ be in the family No. 126. Since $X$ is quasi-smooth, $X$ can be assumed to be defined by:
\[
X_{6}: xw+yt+z^2=0 \subseteq {\mathbbm P}(1_x,2_y,3_z,4_t,5_w), 
\]
Then affine chart $X \cap {\mathbbm D}_+ (x)$ is isomorphic to ${\mathbbm A}_k^3$. 

\subsubsection{{\rm No. 127}}\label{127} 
Let $X$ be in the family No. 127. This case can be dealt with by almost same argument as in \ref{125}, but we need to be more careful because the lowest degree among coordinates of the ambient weighted projective space is greater than $1$. Since $X$ is quasi-smooth, we may assume that $X$ is defined by:
\[
X_{12} : tw+y^4+z(z+x^2)(z-x^2)=0 \subseteq {\mathbbm P}(2_x,3_y,4_z,5_t,7_w). 
\]
The affine chart $X \cap {\mathbbm D}_+ (w)$ is isomorphic to:
\[
{\mathbbm A}_k^3 (x,y,z) / \! / {\bf \mu}_7 (2,3,4) \cong 
{\mathbbm A}_k^3 (x,y,z) / \! / {\bf \mu}_7 (1,5,2), 
\]
which contains an ${\mathbbm A}^2$-cylinder isomorphic to ${\mathbbm A}_k^{2} \times ( {\mathbbm A}_k^1 \backslash \{ o \} )$ by Lemma \ref{lem:cylinder}. On the contrary, $X$ does not contain the affine $3$-space ${\mathbbm A}_k^3$ by virtue of Lemma \ref{lem:cr}.

\subsubsection{{\rm No. 128}}\label{128} 
Let $X$ be in the family No. 128. Since $X$ is quasi-smooth, we may assume that $X$ is defined by:
\[
X_{12}: zw+t^2 +y(y+x^4)(y-x^4)=0 \subseteq {\mathbbm P}(1_x,4_y,5_z,6_t,7_w), 
\]
Then, for instance, the affine chart $X \cap {\mathbbm D}_+ (w)$ is isomorphic to 
\[
{\mathbbm A}_k^3 (x,y,t) / \! / {\bf \mu}_7 (1,4,6), 
\]
which contains an open subset isomorphic to ${\mathbbm A}_k^{2} \times ( {\mathbbm A}_k^1 \backslash \{ o \} )$ by Lemma \ref{lem:cylinder}. Meanwhile, provided that $X$ contains an open subset $U$ such that $U \cong {\mathbbm A}_k^3$, the complement $H:= X \backslash U$ coincides with ${\mathbbm V}_+ (x) \cap X$ by Lemma \ref{lem:cr}, thence $U$ is isomorphic to:
\[
zw+ t^2+y(y+1)(y-1) =0 \, \subseteq \, {\mathbbm A}_k^4={\rm Spec} (k [y,z,t,w]). 
\]
But this hypersurface can not be isomorphic to ${\mathbbm A}_k^3$ since the affine curve defined by $t^2+y(y+1)(y-1) =0$ in ${\mathbbm A}_k^2={\rm Spec}(k[y,t])$ is not rational, see Lemma \ref{lem:am}. 

\subsubsection{{\rm No. 129}}\label{129} 
Let $X$ be in the family No. 129. As $X$ is assumed to be quasi-smooth, $X$ is defined by:
\[
X_{10}: yw+t^2+x(z+x^2)(z-x^2)=0 \subseteq {\mathbbm P}(2_x,3_y,4_z,5_t,7_w). 
\]
Then we look into the affine chart $X \cap {\mathbbm D}_+ (y)$, which is isomorphic to:
\[
 {\mathbbm A}_k^3 (x,z,t) / \! / {\bf \mu}_3 (2,4,5) \cong  {\mathbbm A}_k^3 (x,z,t) / \! / {\bf \mu}_3 (2,1,2), 
\]
which contains an open subset isomorphic to the ${\mathbbm A}^2$-cylinder ${\mathbbm A}_k^{2} \times ( {\mathbbm A}_k^1 \backslash \{ o \} )$ by Lemma \ref{lem:cylinder}. Whereas $X$ does not contain the affine $3$-space in consideration of Lemma \ref{lem:cr}.

\subsubsection{{\rm No. 130}}\label{130} 
Let $X$ be in the family No. 130. Since $X$ is quasi-smooth, $X$ is defined by:
\[
X_{12}: zw+t^2+y^3+x^4=0 \subseteq {\mathbbm P}(3_x,4_y,5_z,6_t,7_w). 
\]
Notice that $X$ never contains the affine $3$-space by Lemma \ref{lem:cr}. On the other hand, we see:
\[
X \cap {\mathbbm D}_+ (z) \cong {\mathbbm A}_k^3 (x,y,t) / \! / {\bf \mu}_5 (3,4,6) 
= {\mathbbm A}_k^3 (x,y,t) / \! / {\bf \mu}_5 (3,4,1),
\]
which contains an ${\mathbbm A}^2$-cylinder isomorphic to ${\mathbbm A}_k^{2} \times ( {\mathbbm A}_k^1 \backslash \{ o \} )$ by Lemma \ref{lem:cylinder}.

\section{Automorphism Groups}
In this section, we classify the automorphism groups of all quasi-smooth weighted hypersurface Fano threefolds in the families $(\spadesuit)$.

\subsection{}\label{3-1}
At first, we will explain briefly our strategy to seek for the full automorphism groups of quasi-smooth weighted hypersurface Fano threefolds in the families $(\spadesuit)$. Let ${\mathbbm P}\coloneqq{\mathbbm P}(a_0,a_1,a_2,a_3,a_4)$ be the weighted projective space with coordinates $x_0,\ldots, x_4$, where the weight of $x_i$ is $a_i$ for each $i$ {\color{black}{such that $a_0 \leqq \cdots \leqq a_4$}}, and let $X\subset {\mathbbm P}$ be a quasi-smooth hypersurface of degree $d$. Suppose that $d\leq 2a_4$. Since $X$ is quasi-smooth, there is an index $i$ such that $d = a_i + a_4$. By a suitable coordinate change we may assume that the hypersurface $X$ is defined by a quasihomogeneous polynomial of the form
\begin{equation*}
    x_i x_4 + f(x_l, x_m, x_n) = 0,
\end{equation*}
where $\{i,l,m,n\} = \{0,1,2,3\}$ and $f(x_l, x_m, x_n)$ is a quasihomogeneous polynomial of degree $d$ that defines a quasi-smooth curve in ${\mathbbm P}(a_l, a_m, a_n)$. 

We consider the projection map $\pi\colon X\dashrightarrow {\mathbbm P}(a_0,a_1,a_2,a_3)$ from the point $\mathsf{p}_4 =(0:0:0:0:1)$. Let $\eta\colon Y\to X$ be the weighted blow-up at the point $\mathsf{p}_4$ with weights: 
\[
{\color{black}{{\rm wt} (x_l, x_m, x_n) = \frac{1}{a_4}(a_l, a_m, a_n)}}
\]
and let $\widetilde{E}$ is the $\eta$-exceptional divisor. We obtain the birational morphism $\tau\colon Y\to {\mathbbm P}(a_0,a_1,a_2,a_3)$ such that the following diagram commutes:
\[\begin{tikzcd}
	Y \\
	\\
	X && {\mathbbm P}(a_0,a_1,a_2,a_3)
	\arrow["\eta"',from=1-1, to=3-1]
	\arrow["\tau",from=1-1, to=3-3]
	\arrow["\pi",dashed, from=3-1, to=3-3]
\end{tikzcd}\]
We consider the case where $a_i< a_4$. For any automorphism $\phi$ of $X$, we have $\phi(\mathsf{p}_4) = \mathsf{p}_4$ and $\phi^*(H) = H$, where $H$ is the hyperplane section defined by $x_i = 0$. Then $\phi$ extend to the automorphism $\phi_{\eta}$ of $Y$ such that $\phi_{\eta}^*(\widetilde{E}) = \widetilde{E}$ and $\phi_{\eta}^*(\widetilde{H}) = \widetilde{H}$, where $\widetilde{E}$ is $\eta$-exceptional and $\widetilde{H}$ is the strict transform of $H$, respectively. Moreover, $\phi_{\eta}$ induce the automorphism $\phi_{\tau}$ of ${\mathbbm P}(a_0,a_1,a_2,a_3)$ such that $\phi_{\tau}^*(E) = E$ and $\phi_{\tau}^*(h) = h$, where $E = \{ x_i = 0 \}$ is the image of $\widetilde{E}$ and $h = \{ x_i = f(x_l, x_m, x_n) = 0 \}$ is the image of $\widetilde{H}$, respectively. Similarly, for every automorphism of ${\mathbbm P}(a_0,a_1,a_2,a_3)$ that {\color{black}{leaves}} both $E$ and $h$, there exists the corresponding automorphism of $X$. Therefore we have: 
\begin{equation}\label{Aut}
    \mathrm{Aut}(X) \cong \mathrm{Aut}({\mathbbm P}(a_0,a_1,a_2,a_3); E, h),
\end{equation}
where $\mathrm{Aut}({\mathbbm P}(a_0,a_1,a_2,a_3); E, h)$ denotes the automorphism group of ${\mathbbm P}(a_0,a_1,a_2,a_3)$ consisting of automorphisms that leave both $E$ and $h$.

Next, we consider the case where $a_i = a_4$. Without loss of generality, we may assume that $i=3$. Every automorphism $\phi$ of $X$ satisfies one of the following two possibilities: either $\phi(\mathsf{p}_4) = \mathsf{p}_4$ and $\phi^*(H) = H$, or $\phi(\mathsf{p}_3) =  \mathsf{p}_4$ and $\phi^*(H_4) = H$, where $\phi(\mathsf{p}_3) = (0:0:0:1:0)$ and $H_4$ denotes the hyperplane section defined by $x_4 = 0$. In the first case, $\phi$ corresponds the automorphism contained in $\mathrm{Aut}({\mathbbm P}(a_0,a_1,a_2,a_3); E, h)$ by the above process. In the second case, we consider the involution $\nu\colon X\to X$ defined by:
\[
(x_0:x_1:x_2:x_3:x_4) \mapsto (x_0:x_1:x_2:x_4:x_3).
\]
Then we have a splitting exact sequence of groups:
\[
1 \longrightarrow {\rm Aut}(X, \mathsf{p}_4) \longrightarrow {\rm Aut} (X) \longrightarrow \big{\langle} \nu \big{\rangle} \longrightarrow 1, 
\]
which says that:
\begin{equation} \label{Aut:Z_2}
{\rm Aut}(X) \cong  {\rm Aut}(X, \mathsf{p}_4) \rtimes \big{\langle} \nu \big{\rangle} \cong  \mathrm{Aut}({\mathbbm P}(a_0,a_1,a_2,a_3); E, h) \rtimes {\mathbbm Z}_2.
\end{equation}


\subsection{}\label{3-2}
In what follows, we will determine the full-automorphism group ${\rm Aut}(X)$ for all quasi-smooth weighted hypersurafce Fano threefolds belonging to the families $(\spadesuit)$.
\subsubsection{{\rm No. 104}}
The Fano threefold in the family No. 104 is uniquely determined, i.e., it is nothing but the smooth quadric hypersurface. It is known that the automorphism group of $X$ is isomorphic to $\mathrm{PSO}_5$ (see \cite[Theorem 1.1.2]{KPS18}). 
\subsubsection{{\rm No. 105}}
Let $X$ be in the family No. 105. By a suitable coordinate change, we may assume that $X$ is defined by:
\begin{equation*}
    X_3: tw + f(x,y,z) = 0 \subset \mathbbm{P}(1_x,1_y,1_z,1_t,2_w),
\end{equation*}
where $f(x,y,z)$ is a homogeneous polynomial of degree $3$ such that the curve defined by $f(x,y,z) = 0$ in $\mathbbm{P}_k^2$ is a smooth cubic plane curve. By applying (\ref{Aut}) we only need to consider the automorphism group $\mathrm{Aut}({\mathbbm P}(1_x,1_y,1_z,1_t); E, h)$, where $E$ is the divisor defined by $t=0$ and $h$ is the curve defined by $t= f(x,y,z) = 0$. Every automorphism $\phi\in\mathrm{Aut}({\mathbbm P}(1_x,1_y,1_z,1_t); E, h)$ is of the form: 
\begin{equation*}
    \phi(x:y:z:t) = (\xi_1x+\xi_2y+\xi_3z+\xi_4 t : \upsilon_1x+\upsilon_2y+\upsilon_3z + \upsilon_4t: \zeta_1x+\zeta_2y+\zeta_3z + \zeta_4t: \tau t),
\end{equation*}
where $\xi_i$, $\upsilon_i$, $\zeta_i$ and $\tau$ are constants. The automorphism $\phi$ can be decomposed into two types of automorphisms, that is, $\phi = \psi_t \circ \psi$, where
\begin{align*}
    \psi(x:y:z:t) &= (\xi_1x+\xi_2y+\xi_3z : \upsilon_1x+\upsilon_2y+\upsilon_3z : \zeta_1x+\zeta_2y + \zeta_3 z : t),\\
    \psi_t(x:y:z:t) &= (x + \xi_4t :y+\upsilon_4t : z+\zeta_4t : \tau t).
\end{align*}
The restriction:
\[
\phi|_E = \psi|_E: (x:y:z) \mapsto \big{(}\xi_1x+\xi_2y+\xi_3z : \upsilon_1x+\upsilon_2y+\upsilon_3z : \zeta_1x+\zeta_2y + \zeta_3 z \big{)}
\]
should be contained in ${\rm Aut}( {\mathbbm P}_k^2, h)$ and vice versa. Thus the sub-group $G$ of $\mathrm{Aut}({\mathbbm P}(1_x,1_y,1_z,1_t); E, h)$ generated by $\psi$ is isomorphic to the group $G= \mathrm{Aut}(\mathbbm{P}_k^2,h)$ composed of linear automorphisms of the smooth plane cubic curve $h$. See \cite{BM} for a classification of such automorphism groups. Therefore, we have
\begin{equation*}
    \mathrm{Aut}(X)\cong (\mathbbm{G}_a^3\rtimes \mathbbm{G}_m)\rtimes G.
\end{equation*}

\subsubsection{{\rm No. 106}}
Let $X$ be in the family No. 106. By a suitable coordinate change, we may assume that $X$ is defined by:
\begin{equation*}
    X_4: tw + f(x,y,z) = 0 \subset \mathbbm{P}(1_x,1_y,1_z,2_t,2_w),
\end{equation*}
where $f(x,y,z)$ is a homogeneous polynomial of degree $4$ such that the curve defined by $f(x,y,z) = 0$ in $\mathbbm{P}_k^2$ is a smooth quartic plane curve. Since $2_t = 2_w$, we will classify $\mathrm{Aut}({\mathbbm P}(1_x,1_y,1_z,2_t); E, h)$ in order to determine the automorphism group $\mathrm{Aut}(X)$ by applying (\ref{Aut:Z_2}), where $E = \{ t=0 \}$ and $h = \{ t = f(x,y,z)=0 \}$. Every automorphism $\phi\in\mathrm{Aut}({\mathbbm P}(1_x,1_y,1_z,2_t); E, h)$ is of the form: 
\begin{equation*}
    \phi(x:y:z:t) = (\xi_1x+\xi_2y+\xi_3z : \upsilon_1x+\upsilon_2y+\upsilon_3z : \zeta_1x+\zeta_2y + \zeta_3 z : \tau t),
\end{equation*}
where $x_i$, $\upsilon_i$, $\zeta_i$ and $\tau$ are constant. The automorphism $\phi$ can be decomposed into two types of automorphisms, that is, $\phi = \psi \circ \psi_t$, where
\begin{align*}
 \psi_t(x:y:z:t) &= (x : y : z : \tau t), \\ 
    \psi(x:y:z:t) &= (\xi_1x+\xi_2y+\xi_3z : \upsilon_1x+\upsilon_2y+\upsilon_3z : \zeta_1x+\zeta_2y + \zeta_3 z : t). 
\end{align*}
The restriction:
\[
\phi|_E=\psi|_E: (x:y:z) \mapsto \big{(} \xi_1x+\xi_2y+\xi_3z : \upsilon_1x+\upsilon_2y+\upsilon_3z : \zeta_1x+\zeta_2y + \zeta_3 z \big{)}
\]
should be contained in ${\rm Aut}({\mathbbm P}(1_x,1_y,1_z),h)$ and vice versa. Therefore the sub-group $G$ of the group $\mathrm{Aut}({\mathbbm P}(1_x,1_y,1_z,2_t); E, h)$ generated by $\psi$ is isomorphic to $\mathrm{Aut}(\mathbbm{P}_k^2,h)$ consisting of linear automorphisms of the smooth plane quartic curve $h$. See \cite[Table 6.1]{Dol} for a classification of such automorphism groups. Therefore we have
\begin{equation*}
    \mathrm{Aut}(X) \cong (\mathbbm{G}_m \times G)\rtimes \mathbbm{Z}_2 \cong {\mathbbm G}_m \rtimes (G \times {\mathbbm Z}_2 ),
\end{equation*}
where ${\mathbbm Z}_2$ acts on ${\mathbbm G}_m$ by means of the inverse action. 
\subsubsection{{\rm No. 111}}\label{Aut:No. 111}
Let $X$ be in the family No. 111. By a suitable coordinate change, we may assume that $X$ is defined by:
\[
X_4: zw + t^2 + f_4(x,y) =0 \subseteq {\mathbbm P}(1_x,1_y,1_z,2_t,3_w), 
\]
where $f_4(x,y)$ is a homogeneous polynomial of degree $4$ in $x$ and $y$, such that $f_4(x,y)=0$ defines four distinct points in ${\mathbbm P}_k^1$. By applying (\ref{Aut}) we only need to consider the automorphism group $\mathrm{Aut}({\mathbbm P}(1_x,1_y,1_z,2_t); E, h)$, where $E$ is the divisor defined by $z=0$ and $h$ is the curve defined by $z= t^2 + f_4(x,y) =0$. Every automorphism $\phi\in\mathrm{Aut}({\mathbbm P}(1_x,1_y,1_z,2_t); E, h)$ is of the form: 
\begin{equation*}
    \phi(x:y:z:t) = (\xi_1x+\xi_2y+\xi_3z : \upsilon_1x+\upsilon_2y+\upsilon_3z : \zeta z : \tau_1 t + \tau_2 xz + \tau_3 zy + \tau_4 z^2),
\end{equation*}
where $\xi_i$, $\upsilon_i$, $\zeta$ and $\tau_i$ are constants and $\tau_1$ is either $-1$ or $1$. The automorphism $\phi$ can be decomposed into three types of automorphisms, that is, $\phi = \varphi \circ \psi_{x,y} \circ \psi_t$, where
\begin{align*}
    \psi_t(x:y:z:t) &= (x : y : z : \tau_1 t), \\ 
    \psi_{x,y}(x:y:z:t) &= (\xi_1x+\xi_2y : \upsilon_1x+\upsilon_2y : z : t),\\
    .\varphi(x:y:z:t) &= \Big{(} x + \xi_3z : y + \upsilon_3z : \zeta z : t + \big{(} \frac{\upsilon_2 \tau_2 -\upsilon_1 \tau_3}{\xi_1 \upsilon_2 - \xi_2 \upsilon_1} \big{)} xz + 
    \big{(} \frac{-\xi_2 \tau_2 +\xi_1 \tau_3}{\xi_1 \upsilon_2 - \xi_2 \upsilon_1} \big{)} zy + \tau_4 z^2 \Big{)}. 
\end{align*}
The subgroup of $\mathrm{Aut}({\mathbbm P}(1_x,1_y,1_z,2_t); E, h)$ generated by automorphisms $\varphi$ is isomorphic to $\mathbbm{G}_a^5 \rtimes \mathbbm{G}_m$. On the other hand, the restriction:
\[
\psi_{x,y}|_E: (x:y:t) \mapsto \big{(} \xi_1 x+ \xi_2 y: \upsilon_1 x+ \upsilon_2 y: t \big{)} 
\]
onto $E=\{ z=0 \} \cong {\mathbbm P}(1_x,1_y,2_t)$ should satisfy:
\[
f_4 \big{(} \xi_1 x+ \xi_2 y, \upsilon_1 x+ \upsilon_2 y ) =f_4 (x,y), 
\]
in other words, $\psi_{x,y}|_{\{ z=t=0\}}$ has to leave $4$ points on ${\rm Proj} (k[y,z])$ defined by $f_4(y,z)=0$ invariant. Thus the sub-group of $\mathrm{Aut}({\mathbbm P}(1_x,1_y,1_z,2_t); E, h)$ corresponding to $\psi_{x,y}$ is isomorphic to {\color{black}{$\mathbb{Z}_2 \times \mathbb{Z}_2$}}. The sub-group corresponding to $\psi_t$ is obviously isomorphic to ${\mathbbm Z}_2$. As a result, we obtain:
\begin{equation*}
    \mathrm{Aut}(X) \cong (\mathbbm{G}_a^5 \rtimes \mathbbm{G}_m) \rtimes {\color{black}{(\mathbbm{Z}_2  \times \mathbbm{Z}_2 \times \mathbbm{Z}_2)}}.
\end{equation*}


\subsubsection{{\rm No. 112}}\label{Aut:No. 112}
Let $X$ be in the family No. 112. We can determine the connected component ${\rm Aut}^0(X)$ without much of difficulty, whereas in order to seek for the full automorphism group ${\rm Aut}(X)$ we need to bear somehow complicated case-by-case observation depending on the defining equation of $X$ in the weighted projective space. By a suitable coordinate change, we may assume that $X$ is defined by:
\[
X_6: tw + z^3 + f_4(x,y)z + f_6(x,y) =0 \subseteq {\mathbbm P}(1_x,1_y,2_z,3_t,3_w), 
\]
where $f_i(x,y)$ are homogeneous polynomials of degree $i$. In order for $X$ to be quasi-smooth, $f_6(x,y)$ must be nonzero. Since $3_t = 3_w$, it follows by (\ref{Aut:Z_2}):
\[
{\rm Aut}(X) \cong {\rm Aut} \Big{(} {\mathbbm P}(1_x,1_y,2_z,3_t); E, h \Big{)} \rtimes {\mathbbm Z}_2,
\]
where $E = \{ t=0 \}$ and $h = \{ t = z^3 + f_4(x,y)z + f_6(x,y)=0 \}$. 
For every automorphism $\phi\in\mathrm{Aut}({\mathbbm P}(1_x,1_y,2_z,3_t); E, h)$, we have $\phi^*(E) = E$. It implies that
\begin{equation*}
    \mathrm{Aut} \big{(} {\mathbbm P}(1_x,1_y,2_z,3_t); E, h \big{)} \cong 
    {\mathbbm G}_m \rtimes \mathrm{Aut} \big{(} {\mathbbm P}(1_x,1_y,2_z); \hat{h} \big{)},
\end{equation*}
so that:
\[
{\rm Aut}(X) \cong {\mathbbm G}_m \rtimes \Big{(} \mathrm{Aut} \big{(} {\mathbbm P}(1_x,1_y,2_z); \hat{h} \big{)} \times {\mathbbm Z}_2 \Big{)}, 
\]
where $\hat{h}$ denotes the curve defined by $z^3 + f_4(x,y)z + f_6(x,y) = 0$ in ${\mathbbm P}(1_x,1_y,2_y)$, and the ${\mathbbm Z}_2$ acts on ${\mathbbm G}_m$ by means of the inverse action. Set $G=\mathrm{Aut}({\mathbbm P}(1_x,1_y,2_z); \hat{h})$. We consider an automorphism $\psi\in G$. Then $\psi$ is should be of the form: 
\begin{equation*}
    (x:y:z)\mapsto (a_1x + b_1y:a_2x+b_2 y : cz),
\end{equation*}
where $a_i$, $b_i$ for $i\in \{1,2\}$ and $c$ are constants. 

We begin with the case where the quartic form $f_4(x,y)$ is identically zero. Since $X$ is quasi-smooth, the equation $f_6(x,y)=0$ defines distinct six points in $\mathbbm{P}_k^1$. By a suitable coordinate change, we may assume that the curve $\hat{h}$ is defined by
\begin{equation*}
    z^3 + f_6(x,y)= z^3+ x (x-y) (x-\lambda y) (x-\mu y) (x -\gamma y) y =0 
\end{equation*}
on $\mathbbm{P}(1_x,1_y,2_z)$, where $\lambda$, $\mu$ and $\gamma$ are constant. 
Thus $\psi$ should satisfy:
\[
c^3=1 \quad {\rm and} \quad f_6(a_1x+b_1y, a_2x+b_2y)=f_6(x,y). 
\]
Hence it follows that $G  \cong {\mathbbm Z}_3 \times \tilde{G}$, where $\tilde{G}$ is a sub-group of ${\rm Aut}({\mathbbm P}_k^1)={\rm PGL}_2$ consisting of automorphisms leaving the six points $\{ 0,1,\lambda, \mu, \gamma, \infty \}$ of ${\mathbbm P}_k^1$ stable. 
Note that the group $\tilde{G}$ are classified as follows depending on the sextic form $f_6(x,y)$ (up to multiplication by non-zero scalars) (cf. \cite[Section 2, (A)]{Bolza}, \cite[Section 3]{SV20} and \cite{Ge74}):
\begin{equation}\label{Aut:degree 6 homogeneous polynomial}
    f_6(x,y) = 
    \begin{dcases}
        \text{generic  case}, & \tilde{G} \cong \text{trivial},\\  
        (x^2 + y^2)(x^2 + \lambda y^2)(x^2 + \mu y^2), & \tilde{G} \cong \mathbbm{Z}_2,\\
        x(x^5 + y^5), & \tilde{G} \cong \mathbbm{Z}_5,\\
        xy(x^2 + y^2)(x^2 + \lambda y^2), & \tilde{G} \cong \mathbbm{D}_4,\\
        (x^3+y^3)(x^3+\lambda y^3), & \tilde{G} \cong \mathbbm{D}_6,\\
        x^6 + y^6, & \tilde{G} \cong \mathbbm{D}_{12},\\
        xy(x^4 + y^4), & \tilde{G} \cong \mathbbm{S}_4,
    \end{dcases}
\end{equation}
where $\lambda$ and $\mu$ are distinct constants not in $\{0,1\}$. Therefore
\begin{equation*}
    \mathrm{Aut}(X) \cong 
    {\mathbbm G}_m \rtimes ( {\mathbbm Z}_2 \times {\mathbbm Z}_3 \times \tilde{G}),
\end{equation*}
where the part ${\mathbbm Z}_2$, that corresponds to the involution $(x:y:z:t:w) \mapsto (x:y:z:w:t)$, acts on ${\mathbbm G}_m$ by means of the inverse action.

Next, we consider the case where $f_4(x, y)$ is nonzero. We may assume that the coefficient $c$ in the automorphism $\psi$ is equal to $1$. Under the action of $\psi$, we have
\begin{equation*}
    \psi^*(z^3 + f_4(x,y)z + f_6(x,y)) = z^3 + f_4(x,y)z + f_6(x,y).
\end{equation*}
Thus we classify the automorphisms $\psi$ such that $\psi^*(f_4(x,y)) = f_4(x,y)$ and $ \psi^*(f_6(x,y)) = f_6(x,y)$.
Furthermore, we may distinguish cases according to whether the discriminant of $f_4(x,y)$ vanishes or not.

We consider the case where $f_4(x,y)$ has nonvanishing discriminant. The automorphism group $F$ of $\mathbbm{P}_k^1$ that fixes the zero locus defined by $f_4(x,y) = 0$ is isomorphic to one of the following: $\mathbbm{Z}_2\times \mathbbm{Z}_2$, $\mathbbm{D}_8$ or $\mathbbm{A}_4$ (see \cite{Yao19}). The quartic form $f_4(x,y)$ is equivalent to $x^4 + \lambda x^2y^2 + y^4$, for some constant $\lambda$. In this case the group $F$ has a normal subgroup isomorphic to $\mathbbm{Z}_2\times \mathbbm{Z}_2$ and is generated by two matrices 
\begin{equation*}
    M_1=
    \begin{bmatrix}
        0&1\\1&0
    \end{bmatrix}
    ,\qquad
    M_2=
    \begin{bmatrix}
        1&0\\0&-1
    \end{bmatrix}
    \in\mathrm{PGL}_2(k).
\end{equation*}
If the nonzero sextic form 
\begin{equation*}
    f_6(x,y) = \zeta_0 x^6 + \zeta_1 x^5y +\zeta_2 x^4y^2 + \zeta_3 x^3y^3 + \zeta_4 x^2y^4 + \zeta_5 xy^5 + \zeta_6 y^6
\end{equation*}
is fixed under the action $M_1$ then the coefficients must satisfy the relations $\zeta_0 = \zeta_6$, $\zeta_1 = \zeta_5$ and $\zeta_2 = \zeta_4$. If the form is fixed under the action $M_2$ then $\zeta_1 = \zeta_3 = \zeta_5 = 0$. 

The condition $12+\lambda^2 \neq 0$ implies that $F\cong \mathbbm{Z}_2\times \mathbbm{Z}_2$ (see \cite{Yao19}). In this case, $G$ is isomorphic to one of the following: the trivial group, $\mathbbm{Z}_2$ or $\mathbbm{Z}_2\times \mathbbm{Z}_2$.

In the case where the group $F$ is isomorphic to $\mathbbm{A}_4$, that is, when $12+\lambda^2 = 0$, the quartic form $f_4(x,y)$ fixed by $F$ is equivalent to $x(x^3+y^3)$. We denote by $\xi_i$ a primitive $i$th root of unity. Let 
\begin{equation*}
    N=
    \begin{bmatrix}
        1 & 0\\ 0&\xi_3
    \end{bmatrix}
    \in F
\end{equation*}
be the element of order three. If the sextic form $f_6(x,y)$ is not invariant under the action of $N$ then $G$ is isomorphic to one of the following: the trivial group, $\mathbbm{Z}_2$ or $\mathbbm{Z}_2\times \mathbbm{Z}_2$. Therefore we may restrict to the case where the nonzero sextic form $f_6(x,y)$ is invariant under the action of $N$. In this case, $G$ contains a subgroup isomorphic to $\mathbbm{Z}_3$, and $f_6(x,y)$ is of the form
\begin{equation*}
    \zeta_0x^6 + \zeta_3x^3y^3 + \zeta_6y^6,
\end{equation*} 
where $\zeta_0$, $\zeta_3$ and $\zeta_6$ are constants. Since the curve $\hat{h}$ is quasi-smooth, the constant $\zeta_6$ is nonzero. For the constant $\zeta_0$, note that there is no involution in $F$ that fixes $y$, the condition $\zeta_0=0$ implies that $G\cong \mathbbm{Z}_3$. Hence, we focus on the case $\zeta_0\neq 0$. It implies that the sextic form $f_6(x,y)$ admits no linear factor invariant under the action of $N$. Then $f_6(x,y)$ is of the form
\begin{equation*}
    \zeta_0(x+ay)(x+a\xi_3y)(x+a\xi_3^2y)(x+by)(x+b\xi_3y)(x+b\xi_3^2y),
\end{equation*}
where $a$ and $b$ are nonzero constants. If $G$ contains an involution $M$ then $G\cong \mathbbm{A}_4$ and the image $x+cy$ of the linear form $x+ay$ under the action of $M$ coincides with one of the linear factors of $f_6(x,y)$. In the case where $c\in \{b, b\xi_3, b\xi_3^2\}$, the sextic form $f_6(x,y)$ is given by the following: Let $M\in F$ be an involution, and let $N\in F$ be an element of order three. For a nonzero vector $(\alpha, \beta)\in k^2$, the sextic form
\begin{equation*}
    \prod_{k=0}^2 
    \left(
    \begin{bmatrix}
        x&y
    \end{bmatrix}
    N^k
    \begin{bmatrix}
        \alpha\\
        \beta
    \end{bmatrix}
    \right)
    \left(
    \begin{bmatrix}
        x&y
    \end{bmatrix}
    N^kM
    \begin{bmatrix}
        \alpha\\
        \beta
    \end{bmatrix}
    \right),
\end{equation*}
is invariant under group generated by $M$ and $N$. In the case where $c\in \{a,a\xi_3,a\xi_3^2\}$, the sextic form $f_6(x,y)$ is given by the following: Let $(\alpha_1,\beta_1)$ and  $(\alpha_2,\beta_2)$ be nonzero vectors in $k^2$ that are fixed under the action of $M$. Then the sextic form
\begin{equation*}
    \prod_{k=0}^2 
    \left(
    \begin{bmatrix}
        x&y
    \end{bmatrix}
    N^k
    \begin{bmatrix}
        \alpha_1\\
        \beta_1
    \end{bmatrix}
    \right)
    \left(
    \begin{bmatrix}
        x&y
    \end{bmatrix}
    N^k
    \begin{bmatrix}
        \alpha_2\\
        \beta_2
    \end{bmatrix}
    \right),
\end{equation*}
is invariant under group generated by $M$ and $N$.

In the case where the group $F$ is isomorphic to $\mathbbm{D}_8$, the quartic form $f_4(x,y)$ is equivalent to $x^4 + y^4$. Then $F$ is generated by
\begin{equation*}
    M_1=
    \begin{bmatrix}
        0&1\\1&0
    \end{bmatrix}
    , \qquad
    M_2=
    \begin{bmatrix}
        1&0\\0&-i
    \end{bmatrix}
    \in\mathrm{PGL}_2(k).
\end{equation*}
The sextic form $f_6(x,y)$ that is fixed under the action of $M_2$ is of the form:
\begin{equation*}
    \zeta_0x^6 + \zeta_4x^2y^4
\end{equation*}
where $\zeta_0$ and $\zeta_4$ are constants. These forms are not fixed by $M_1$. Thus $G\cong \mathbbm{Z}_4$. Otherwise $G$ is a subgroup of $\mathbbm{Z}_2\times \mathbbm{Z}_2$. This has already been considered.

If the discriminant of $f_4(x,y)$ vanishes, then $f_4(x,y)$ can be taken to be one of the following forms by a suitable coordinate change, : $xy(x+y)^2$, $x^2y^2$, $x^3y$ or $x^4$. 

We consider the case where $f_4(x,y)$ is of the form $xy(x+y)^2$. Then the only nontrivial automorphism $\psi_{xy}$ is given by $(x:y:z)\mapsto (y:x:z)$. Set
\begin{equation*}
    f_6(x,y)=\zeta_0x^6 + \zeta_1x^5y + \zeta_2x^4y^2 + \zeta_3x^3y^3 + \zeta_4x^2y^4 + \zeta_5xy^5 + \zeta_6y^6,
\end{equation*}
where $\zeta_i$ are constants. In order for the sextic form $f_6(x,y)$ to be fixed under the action of $\psi_{xy}$, the symmetry condition $\zeta_i = \zeta_{6-i}$ holds for $i=0,1,2$. Then $G\cong \mathbbm{Z}_2$, otherwise, $G$ is trivial.

We consider the case where $f_4(x,y)$ is of the forms $x^2y^2$. The quasi-smoothness of the curve $\hat{h}$ implies that at least one of the coefficients $\zeta_0$ or $\zeta_1$, and at least one of $\zeta_5$ or $\zeta_6$, in the sextic form $f_6(x,y)$, must be nonzero. The automorphism group $G$ is generated by the automorphisms $\psi_{a,b}$ defined by $(x:y:z)\mapsto (ax:by:z)$, where $a$ and $b$ are nonzero constants satisfying certain conditions specified below, possibly together with the automorphism $\psi_{xy}$ defined by $(x:y:z)\mapsto (y:x:z)$. 

We first determine the subgroup $G_1$ of $G$ generated by the automorphisms $\psi_{a,b}$. When only $\zeta_0$ and $\zeta_6$ are nonzero, conditions $a^2b^2 = a^6 = b^6 = 1$ hold. It implies that $G_1$ is generated by $\psi_{\xi_6, \xi_6^2}$ and $\psi_{1, -1}$. Consequently, $G_1\cong\mathbbm{Z}_6\times \mathbbm{Z}_2$. When only $\zeta_1$ and $\zeta_5$ are nonzero, condtions $a^2b^2 = a^5b = ab^5 = 1$ hold. It implies that $G_1$ is generated by $\psi_{\xi_8, \xi_8^3}$. Consequently, $G_1\cong\mathbbm{Z}_8$. When only $\zeta_0$ and $\zeta_5$ are nonzero, the conditions $a^2b^2 = a^6 = ab^5 = 1$. Then $(a,b)$ is either $(1,1)$ or $(-1,-1)$. It implies that $G_1\cong\mathbbm{Z}_2$. Similarly we can see that if only $\zeta_1$ and $\zeta_6$ are nonzero, $G_1\cong\mathbbm{Z}_2$. If we additionally assume that $\zeta_3\neq 0$ then $G_1$ is isomorphisc to one of the following groups: $\mathbbm{Z}_6$, $\mathbbm{Z}_4$ or $\mathbbm{Z}_2$. When $\zeta_2$ or $\zeta_4$ are nonzero, $a^2=b^2=1$ hold. Under the conditions $\zeta_0$ and $\zeta_6$ are nonzero, we have $G_1\cong \mathbbm{Z}_2\times \mathbbm{Z}_2$. In the remaining cases, $G_1\cong \mathbbm{Z}_2$.

Next, we consider the automorphism $\psi_{xy}$. The sextic form $f_6(x,y)$ is fixed under the action of $\psi_{xy}$, if and only if its coefficients satisfy the symmetry condition $\zeta_i = \zeta_{6-i}$ for $i=0,1,2$. Consequently, if this symmetry condition holds then $G\cong G_1\times \mathbbm{Z}_2$, otherwise $G\cong G_1$.

We consider the case where $f_4(x,y)$ is of the forms $x^3y$. Then we have $b_1 = a_2 = 0$ that are the constants in the automorphism $\psi$. Then $a_1^3b_2=1$ holds. Since the curve $\hat{h}$ is quasi-smooth, the sextic form $f_6(x,y)$ contains monomials $x^6$ or $x^5y$ and also contains monomials $xy^5$ or $y^6$. It implies that one of the following hold:
\begin{equation*}
    1=a_1^6 = a_1b_2^5,\qquad 1=a_1^6 = b_2^6,\qquad 1=a_1^5b_2 = a_1b_2^5,\qquad 1=a_1^5b_2 = b_2^6.
\end{equation*}
In each case, the group $G$ is isomorphic to $\mathbbm{Z}_2$, $\mathbbm{Z}_6$, $\mathbbm{Z}_2$, $\mathbbm{Z}_2$, respectively.

The remaining case is when the quartic form $f_4(x,y)$ is of the form $x^4$. In this case, $f_6(x,y)$ is either of the form 
\begin{equation*}
    y^6 + c_2x^2y^4 + c_3x^3y^3 + c_4x^4y^2 + c_5x^5y + c_6x^6
\end{equation*}
or
\begin{equation*}
    xy^5 +  d_3x^3y^3 + d_4x^4y^2 + d_5x^5y + d_6x^6.
\end{equation*}
Then the coefficients $b_1$ and $a_2$ in the automorphism $\psi$ must vanish. We may assume that $c = 1$ from the beginning. Then it follows that $a_1^4 =1$ and $f_6(a_1x, b_2y) \equiv f_6(x,y)$. We consider the case where:
\[
f_6(x,y) = y^6 + c_2x^2y^4 + c_3x^3y^3 + c_4x^4y^2 + c_5x^5y + c_6x^6. 
\]
Then the pair $(a_1, b_2)$ has to satisfy the the equations:
\[
a_1^4=1, \, b_2^6=1, \, a_1^2b_2^4c_2=c_2, \, a_1^3b_2^3c_3=c_3, \, a_1^4b_2^2c_4=c_4, \, a_1^5b_2 c_5=c_5, \, a_1^6 c_6=c_6. 
\]
In the case of $c_2=c_3=c_4=c_5=c_6=0$, the corresponding group, which is generated by $(\xi_4, 1)$ and $(1, \xi_6)$, is isomorphic to ${\mathbbm Z}_4 \times {\mathbbm Z}_6$, where $\xi_4$ and $\xi_6$ are respectively primitive fourth root of unity and primitive sixth root of unity. This is the case where the corresponding group becomes the largest. Depending on whether the coefficients of $c_2, c_3, c_4, c_5$ and $c_6$ are zero or not, the group $G$ varies. For instance, if one of the follwoing cases occurs:
\[
(c_2,c_3,c_4,c_5,c_6)= (\ast, 0,0,0,0), \, (\ast, 0, \ast, 0,0), \, (\ast, 0, 0, 0, \ast), 
\]
then $G$ is isomorphic to the Klein four group ${\mathbbm Z}_2 \times {\mathbbm Z}_2$, where $\ast$ stands for any non-zero scalar. The straightforward computation allows to conclude that by means of the suitable choice of coefficients $(c_2, c_3, c_4, c_5, c_6)$, we can realize sub-groups of ${\mathbbm Z}_4 \times {\mathbbm Z}_6$, which are isomorphic to:
\[
{\mathbbm Z}_4 \times {\mathbbm Z}_6, \, {\mathbbm Z}_2 \times {\mathbbm Z}_6, \, {\mathbbm Z}_4 \times {\mathbbm Z}_2, \, {\mathbbm Z}_6, \, {\mathbbm Z}_2 \times {\mathbbm Z}_2, \, {\mathbbm Z}_2, 
\]
as $G$. In other words, any sub-group of ${\mathbbm Z}_4 \times {\mathbbm Z}_6$ isomorphic to either ${\mathbbm Z}_4$, ${\mathbbm Z}_3$ and the trivial one can not be realized as $G$. Indeed, even if all the coefficients $c_2,c_3, c_4, c_5$ and $c_6$ are non-zero, the group $G$ is isomorphic to ${\mathbbm Z}_2$. 

We finally deal with the case where:
\[
f_6(x,y) = xy^5 +  d_3x^3y^3 + d_4x^4y^2 + d_5x^5y + d_6x^6.
\]
Then the pair $(a_1, b_2)$ should satisfy the following equations:
\[
a_4=1, \, a_1b_2^5=1, \, a_1^3b_1^3 d_3=d_3, \, a_1^4b_2^2d_4=d_4, \, 
a_1^5b_2d_5=d_5, \, a_1^6 d_6 =d_6. 
\]
For instance, if $d_3=d_4=d_5=d_6=0$, then it follows that the corresponding group $G$ is isomorphic to the cyclic group ${\mathbbm Z}_{20}$. This is the case in which the group $G$ becomes the biggest. By the straightforward computation, we can actually see that all sub-groups of ${\mathbbm Z}_{20}$ other than the one isomorphic to ${\mathbbm Z}_5$ can be realized as $G$. For example, if $d_6 \neq 0$ and $d_3=d_4=d_5=0$ (resp. $d_5 \neq 0$ and $d_3=d_4=d_6=0$), then $G$ is isomorphic to ${\mathbbm Z}_{10}$ (resp. ${\mathbbm Z}_4$). 

\subsubsection{{\rm No. 113}}
Let $X$ be in the family No. 113. Then in consideration of the quasi-smoothness of $X$, we may assume that $X$ is defined by:
\[
X_4: yw + zt + x^4 =0 \subseteq {\mathbbm P}(1_x,1_y,2_z,2_t,3_w). 
\]
By applying (\ref{Aut}) we only need to consider the automorphism group $\mathrm{Aut}({\mathbbm P}(1_x,1_y,2_z,2_t); E, h)$, where $E$ is the divisor defined by $y=0$ and $h$ is the curve defined by $y= zt + x^4 =0$. The straightforward computation says that an automorphism $\phi \in \mathrm{Aut}({\mathbbm P}(1_x,1_y,2_z,2_t))$ belongs to the sub-group $\mathrm{Aut}({\mathbbm P}(1_x,1_y,2_z,2_t); E, h)$ if and only if it is either of the form:
\begin{equation*}
    \phi(x:y:z:t) = (\xi_1 x+ \xi_2 y : \upsilon y : \zeta_1 z + \zeta_2 xy + \zeta_3 y^2 : \zeta_1^{-1}t + \tau_1 xy + \tau_2 y^2), \, {\rm or}
\end{equation*}
\begin{equation*}
    \phi(x:y:z:t) = (\xi_1 x+ \xi_2 y : \upsilon y : \zeta_1 t + \zeta_2 xy + \zeta_3 y^2 : \zeta_1^{-1}z + \tau_1 xy + \tau_2 y^2), 
\end{equation*}
where $\xi_1 \in k^\times$ is such that $\xi_1^4=1$, $\upsilon, \zeta_1 \in k^\times$, and $\xi_2, \zeta_2, \zeta_3, \tau_1, \tau_2 \in k$. We will observe the former case. Then $\phi$ can be realized as a composite $\phi= \varphi_3 \circ \varphi_2 \circ \varphi_1$, where:
\begin{align*}
    \varphi_1(x:y:z:t) &= \big{(} \xi_1 x : y : z : t \big{)},\\
    \varphi_2(x:y:z:t) &= \big{(} x : y : \zeta_1 z : \zeta_1^{-1} t \big{)},\\
    \varphi_3(x:y:z:t) &= \Big{(} x+\xi_2 y : \upsilon y : z+ \xi_1^{-1}\zeta_2 xy + \zeta_3 y^2 : t + \xi_1^{-1} \tau_1 xy + \tau_2 y^2 \Big{)}.
\end{align*}
Therefore it follows that:
\[
\mathrm{Aut}(X) \cong 
\mathrm{Aut}({\mathbbm P}(1_x,1_y,2_z,2_t); E, h) \cong \big{(} {\mathbbm G}_a^5 \rtimes {\mathbbm G}_m^2 \big{)} \rtimes ({\mathbbm Z}_4 \times {\mathbbm Z}_2) 
\]
where the ${\mathbbm Z}_2$ part is generated by the involution:

\begin{equation*}
    (x:y:z:t) \mapsto (x:y:t:z), 
\end{equation*}
and it acts on ${\mathbbm G}_m^2$ by means of:
\begin{equation*}
    (\lambda, \mu) \mapsto (\lambda, \mu^{-1}). 
\end{equation*}

\subsubsection{{\rm No. 114}} 
Let $X$ be in the family No. 114. As $X$ is assumed to be quasi-smooth, we may assume that $X$ is defined by:
\[
X_6: zw+t^2+f_6(x,y)=0  \subseteq {\mathbbm P}(1_x,1_y,2_z,3_t,4_w), 
\]
where $f_6(x,y)$ is a homogeneous polynomial of degree $6$ such that $f_6(x,y)=0$ yields distinct six points on ${\mathbbm P}_k^1$. By applying (\ref{Aut}) we only need to consider the automorphism group $\mathrm{Aut}({\mathbbm P}(1_x,1_y,2_z,3_t); E, h)$, where $E$ is the divisor defined by $z=0$ and $h$ is the curve defined by $z= t^2 + f_6(x,y) =0$. Every automorphism $\phi\in\mathrm{Aut}({\mathbbm P}(1_x,1_y,2_z,3_t); E, h)$ should be of the form
\begin{equation*}
    \phi(x:y:z:t) = (\xi_1x+\xi_2y : \upsilon_1x+\upsilon_2y : \zeta z : \tau_1 t + \tau_2 xz + \tau_3 yz),
\end{equation*}
where $\xi_i$, $\upsilon_i$, $\zeta$ and $\tau_i$ are constants, $\zeta \neq 0$ and $\tau_1$ is either $-1$ or $1$. The automorphism $\phi$ can be decomposed into three types of automorphisms, that is, $\phi = \varphi \circ \psi_{x,y} \circ \psi_t$, where
\begin{align*}
\psi_t(x:y:z:t) &= (x : y : z : \tau_1 t), \\ 
 \psi_{x,y}(x:y:z:t) &= (\xi_1x+\xi_2y : \upsilon_1x+\upsilon_2y : z : t), \\ 
    \varphi(x:y:z:t) &= \Big{(} x : y : \zeta z : t + \big{(} \frac{\upsilon_2 \tau_2 - \upsilon_1 \tau_3}{\xi_1 \upsilon_2 - \xi_2 \upsilon_1} \big{)} xz + \big{(} \frac{ -\xi_2 \tau_2 + \xi_1 \tau_3}{\xi_1 \upsilon_2 - \xi_2 \upsilon_1} \big{)} yz \Big{)}. 
\end{align*}
The subgroup of $\mathrm{Aut}({\mathbbm P}(1_x,1_y,2_z,3_t); E, h)$ generated by automorphisms corresponding to $\varphi$ is isomorphic to $\mathbbm{G}_a^2 \rtimes \mathbbm{G}_m$. Since the subgroup $G$ of $\mathrm{Aut}({\mathbbm P}(1_x,1_y,2_z,3_t); E, h)$ corresponding to $\psi_{x,y}$ is isomorphic to the group ($\ref{Aut:degree 6 homogeneous polynomial}$), we have
\begin{equation*}
    \mathrm{Aut}(X)\cong (\mathbbm{G}_a^2 \rtimes \mathbbm{G}_m)\rtimes (G\times \mathbbm{Z}_2). 
\end{equation*}
\subsubsection{{\rm No. 115}} 
Let $X$ be in the family No. 115. By a suitable coordinate change we may assume that  $X$ is defined by:
\[
X_6: tw + yz^2 + y^2z + ax^2yz + x^4(by + cz) + dx^6 = 0   \subseteq {\mathbbm P}(1_x,2_y,2_z,3_t,3_w), 
\]
where $a$, $b$, $c$ and $d$ are constants. By applying (\ref{Aut:Z_2}), we see that
\[
{\rm Aut}(X) \cong \mathrm{Aut} \big{(} {\mathbbm P}(1_x,2_y,2_z,3_t); E, h \big{)} \rtimes {\mathbbm Z}_2, 
\]
where $E$ is the divisor defined by $t=0$ and $h$ is the curve defined by $t= yz^2 + y^2z + ax^2yz + x^4(by + cz) + dx^6 = 0$, and the ${\mathbbm Z}_2$-part is generated by the involution:
\[
[x:y:z:t:w] \mapsto [x:y:z:w:t]. 
\]
We will determine the group structure of ${\rm Aut} ({\mathbbm P}(1_x,2_y,2_z,3_t), E,h)$. A given automorphism $\phi\in\mathrm{Aut}({\mathbbm P}(1_x,1_y,2_z,2_t); E, h)$ should be of the form: 
\begin{equation*}
    \phi(x:y:z:t) = (x : \upsilon_1 y + \upsilon_2 z + \upsilon_3 x^2 : \zeta_1 y + \zeta_2 z + \zeta_3 x^2: \tau t ),
\end{equation*}
where $\upsilon_i$, $\zeta_i$ and $\tau \neq 0$ are constants. The automorphism $\phi$ can be decomposed into three types of automorphisms, that is, $\phi = \varphi_3 \circ \varphi_2 \circ \varphi_1$, where
\begin{align*}
    \varphi_1(x:y:z:t) &= (x : y  : z : \tau t ),\\
    \varphi_2(x:y:z:t) &= \big{(} x : \upsilon_1 y + \upsilon_2 z : \zeta_1 y + \zeta_2 z : t \big{)}, \\
    \varphi_3(x:y:z:t) &= \big{(} x: y + \upsilon_3 x^2: z+ \zeta_3 x^2: t \big{)}. 
\end{align*}

\vspace{2mm} 
\noindent {\it Claim.} \, $\upsilon_3 = \zeta_3 =0$.

\vspace{2mm} 
\noindent {\it Proof of Claim.} The restriction $\phi|_E$ of $\varphi$ onto $E=\{ t=0 \} \cong {\mathbbm P}(1_x,2_y,2_z)$ should leave the following curve $h$ stable:
\[
yz(y+z) +a x^2yz + x^4(by+cz) + dx^6=0,  
\]
in particular, $(\upsilon_3, \zeta_3)$ has to satisfy the quadric equation:
\[
\upsilon_3 \zeta_3 (\upsilon_3 + \zeta_3) +a \upsilon_3 \zeta_3 + (b\upsilon_3 + c\zeta_3) =0 
\]
in the unipotent group ${\mathbbm G}_a^2= ( {\rm Spec}(k [y,z]), + )$. This implies in turn that $(\upsilon_3, \zeta_3)=(0,0)$ is the unique solution. 
$\Box$

\vspace{2mm} 
Thus the sub-group generated by automorphism $\varphi_3$ is trivial. On the other hand, the sub-group generated by $\varphi_1$ is isomorphic to ${\mathbbm G}_m$. Now we will look into the sub-group $G$ generated by automorphisms of type $\varphi_2$. The restriction:
\[
\varphi_2|_E : (x:y:z) \mapsto \big{(} x: \upsilon_1 y + \upsilon_2 z: \zeta_1 y + \zeta_2 z \big{)}
\]
leaving the curve $h$ invariant, we have three equations:
\[
\varphi_2^\ast \big{(} yz(y+z) \big{)}= yz(y+z), \, 
\varphi_2^\ast \big{(} ayz \big{)} = ayz, \, 
\varphi_2^\ast \big{(} by+cz \big{)} = by+cz. 
\]
If $a=b=c=0$, then we have only to take care of the first equation $\varphi_2^\ast \big{(} yz(y+z) \big{)}= yz(y+z)$, which implies that $G$ is isomorphic to the symmetric group ${\mathbbm S}_3$. On the other hand, the direct computation says that if:
\[
(a,b,c)= (a,b, \pm b), \, (0, b ,2b), \, (0,b, \frac{ \, 1 \, }{2}b) \quad (b\neq 0), 
\]
then $G$ is isomorphic to ${\mathbbm Z}_2$. Otherwise, $G$ is trivial. To summarize, we have:
\begin{equation*}
    \mathrm{Aut}(X)\cong \mathbbm{G}_m \rtimes ( G \times {\mathbbm Z}_2) \cong 
    \begin{dcases}
        {\mathbbm G}_m \rtimes ( {\mathbbm S}_3  \times \mathbbm{Z}_2 ) & \text{if $a=b=c=0$},\\
        {\mathbbm G}_m \rtimes ({\mathbbm Z}_2 \times \mathbbm{Z}_2) & \text{if $(a,b,c)=(a,b, \pm b), (0, b, 2b), (0, b, \frac{1}{2}b)$ $(b\neq 0)$},\\
        {\mathbbm G}_m \rtimes \mathbbm{Z}_2 & \text{otherwise},\\
    \end{dcases}
\end{equation*}
where the ${\mathbbm Z}_2$-part acts on ${\mathbbm G}_m$ by means of the inverse action. 


\subsubsection{{\rm No. 118}}
Let $X$ be in the family No. 118. Since $X$ is assume to be quasi-smooth, we may assume that $X$ is defined as follows:
\[
X_6:  yw+t^2+ z^3 + ax^4z + bx^6=0  \subseteq {\mathbbm P}(1_x,1_y,2_z,3_t,5_w), 
\] 
where $a$ and $b$ are constants such that at least one of $a$ and $b$ is non-zero. By applying (\ref{Aut}), we only need to determine the automorphism group $\mathrm{Aut}({\mathbbm P}(1_x,2_y,2_z,3_t); E, h)$, where $E$ is the divisor defined by $y=0$ and $h$ is the curve defined by $y=t^2 +  z^3 + ax^4z + bx^6 = 0$. Every automorphism $\phi\in\mathrm{Aut}({\mathbbm P}(1_x,1_y,2_z,3_t); E, h)$ should be of the form
\begin{equation*}
    \phi(x:y:z:t) = (\xi_1 x + \xi_2 y:  \upsilon y : \zeta_1 z +  \zeta_2 xy+ \zeta_3 y^2 : \tau_1 t + \tau_2 yz + \tau_3 x^2y + \tau_4 xy^2 + \tau_5 y^3),
\end{equation*}
where $\xi_i$, $\upsilon$, $\zeta_i$ and $\tau_i$ are constants such that $\upsilon \neq 0$. The automorphism $\phi$ can be decomposed into three types of automorphisms, that is, $\phi = \varphi \circ \psi_{x,z} \circ \psi_t$, where
\begin{align*}
\psi_t(x:y:z:t) &= (x : y : z : \tau_1 t), \\
    \psi_{x,z}(x:y:z:t) &= (\xi_1 x : y: \zeta_1 z : t ),\\
    \varphi(x:y:z:t) &= (x + \xi_2 y:  \upsilon y : z +  \xi_1^{-1} \zeta_2 xy+ \zeta_3 y^2 : t + \xi_1^{-1} \tau_2 yz + \xi_1^{-1} \tau_3 x^2y + \xi_1^{-1} \tau_4 xy^2 + \tau_5 y^3). 
\end{align*}
We will determine the sub-group $G$ generated by $\psi_{x,z}$'s and $\psi_t$'s. The restriction of $\phi$ onto $E=\{ y=0 \} \cong {\mathbbm P}(1_x, 2_z, 3_t)$ yields an automorphism:
\[
{\phi}|_E: \big{(} x:z:t \big{)} \mapsto \big{(} \xi_1 x: \zeta_1 z: \tau_1 t \big{)}
\]
We consider the condition by which $\phi|_E$ leaves the curve defined by $t^2+z^3+ax^4z+bx^6=0$ invariant depending on the pair $(a,b)$. In case of $a\neq 0, b\neq 0$, we have:
\[
\tau_1^2=1, \zeta_1^3=1, \xi_1^4 \zeta_1 =1, \xi_1^6=1, 
\]
which gives rise to $\tau_1 \in \{ \pm 1 \}$ and the solutions for the pair $(\xi_1, \zeta_1)$ generates a sub-group isomorphic to ${\mathbbm Z}_2 \times {\mathbbm Z}_3$. Thus $G \cong {\mathbbm Z}_2^2 \times {\mathbbm Z}_3$. Next in case of $a\neq 0, b=0$, we have:
\[
\tau_1^2=1, \zeta_1^3=1, \xi_1^4 \zeta_1 =1, 
\]
which allows us to conclude that $\tau_1 \in \{ \pm 1 \}$ and the pair $(\xi_1, \zeta_1)$ generates a sub-group isomorphic to ${\mathbbm Z}_3 \times {\mathbbm Z}_4$, so that $G \cong {\mathbbm Z}_2 \times {\mathbbm Z}_3 \times {\mathbbm Z}_4$. Finally, in case of $a=0, b\neq 0$, we have:
\[
\tau_1^2=1, \zeta_1^3=1, \xi_1^6=1, 
\]
which yields a sub-group isomorphic to $G \cong {\mathbbm Z}_2 \times {\mathbbm Z}_3 \times {\mathbbm Z}_6 \cong {\mathbbm Z}_2^2 \times {\mathbbm Z}_3^2$. As a consequence, it follows that:
\begin{equation*}
    \mathrm{Aut}(X)\cong (\mathbbm{G}_a^7 \rtimes \mathbbm{G}_m)\rtimes
    \begin{dcases}
        \mathbbm{Z}_2^2 \times \mathbbm{Z}_3 & \text{if $a\neq 0$ and $b\neq 0$},\\
        \mathbbm{Z}_2 \times {\mathbbm Z}_3 \times \mathbbm{Z}_4  & \text{if $a \neq 0$ and $b=0$},\\
        \mathbbm{Z}_{2}^2 \times \mathbbm{Z}_3^2 & \text{if $a=0$ and $b\neq 0$},\\
    \end{dcases}
\end{equation*}

\subsubsection{{\rm No. 119}}
Let $X$ be in the family No. 119. By a suitable coordinate change, we may assume that $X$ is defined as follows:
\begin{equation*}
    X_6: xw + t^2 + yz(y+z) = 0\subset \mathbbm{P}(1_x,2_y,2_z,3_t,5_w).
\end{equation*}
By applying (\ref{Aut}) we only need to consider the automorphism group $\mathrm{Aut}({\mathbbm P}(1_x,2_y,2_z,3_t); E, h)$, where $E$ is the divisor defined by $x=0$ and $h$ is the curve defined by $x=t^2 + yz(y+z) = 0$. Every automorphism $\phi\in\mathrm{Aut}({\mathbbm P}(1_x,2_y,2_z,3_t); E, h)$ is of the form
\begin{equation*}
    \phi(x:y:z:t) = (\xi x : \upsilon_1 y + \upsilon_2z + \upsilon_3x^2 : \zeta_1z + \zeta_2y + \zeta_3x^2: \tau_1t + \tau_2xz + \tau_3xy + \tau_4x^3),
\end{equation*}
where $\xi$, $\upsilon_i$, $\zeta_i$ and $\tau_i$ are constants satisfying $\tau_1^2 = 1$. The automorphism $\phi$ can be decomposed into three types of automorphisms, that is, $\phi = \varphi \circ \psi_{y,z} \circ \psi_t$, where
\begin{align*}
\psi_t(x:y:z:t) &= (x : y : z : \tau_1t), \\
    \psi_{y,z}(x:y:z:t) &= (x : \upsilon_1 y + \upsilon_2z: \zeta_1z + \zeta_2y : t),\\
     \varphi(x:y:z:t) &= \Big{(} \xi x : y + \upsilon_3 x^2 : z + \zeta_3 x^2 : t + \big{(} \frac{\upsilon_1 \tau_2 - \upsilon_2 \tau_3}{\zeta_1 \upsilon_1 -\zeta_2 \upsilon_2} \big{)} xz + \big{(} \frac{ -\zeta_2 \tau_2 +\zeta_1 \tau_3 }{\zeta_1 \upsilon_1 -\zeta_2 \upsilon_2} \big{)}   xy + \tau_4x^3 \Big{)}. 
\end{align*}
It is easy to see that the subgroup of $\mathrm{Aut}({\mathbbm P}(1_x,2_y,2_z,3_t); E, h)$ corresponding to $\psi_{y,z}$ is isomorphic to $\mathbbm{S}_3$. Therefore, we have
\begin{equation*}
    \mathrm{Aut}(X)\cong (\mathbbm{G}_a^5\rtimes \mathbbm{G}_m) \rtimes (\mathbbm{S}_3\times \mathbbm{Z}_2).
\end{equation*}

\subsubsection{{\rm No. 120}}
Let $X$ be in the family No. 120. By a suitable coordinate change, we may assume that $X$ is defined as follows:
\begin{equation*}
    X_6: yw + zt + x^6 = 0\subset \mathbbm{P}(1_x,2_y,3_z,3_t,4_w).
\end{equation*}
By applying (\ref{Aut}) we have only to consider the automorphism group $\mathrm{Aut}({\mathbbm P}(1_x,2_y,3_z,3_t); E, h)$, where $E$ is the divisor defined by $y=0$ and $h$ is the curve defined by $y=zt + x^6 = 0$. By a straightforward computation, we see that every automorphism $\phi\in\mathrm{Aut}({\mathbbm P}(1_x,2_y,3_z,3_t); E, h)$ is either of the form: 
\begin{equation*}
    \phi(x:y:z:t) = (\xi x : \upsilon y : \zeta_1z + \zeta_2xy: \zeta_1^{-1}t + \tau xy), \, {\rm or}
\end{equation*}
\begin{equation*}
    \phi(x:y:z:t) = (\xi x : \upsilon y : \zeta_1 t + \zeta_2xy: \zeta_1^{-1}z + \tau xy),
\end{equation*}
where $\xi$, $\upsilon$, $\zeta_i$ and $\tau$ are constants such that $\xi, \zeta_1 \in k^\times$ satisfying $\xi^6 = 1$. The automorphism $\phi$ can be decomposed into three types of automorphisms, that is, $\phi = \varphi_3 \circ \varphi_2 \circ \varphi_1$, where 
\begin{align*}
    \varphi_1(x:y:z:t) &= (\xi x : y : z : t ),\\
    \varphi_2(x:y:z:t) &= (x : y: \zeta_1z : \zeta_1^{-1}t ),\\
    \varphi_3(x:y:z:t) &= (x : \upsilon y : z+ \xi^{-1} \zeta_2 xy : t + \xi^{-1}\tau xy)
\end{align*}
for the former case, and 
\begin{align*}
    \varphi_1(x:y:z:t) &= (\xi x : y : z : t ),\\
    \varphi_2(x:y:z:t) &= (x : y: \zeta_1t : \zeta_1^{-1}z ),\\
    \varphi_3(x:y:z:t) &= (x : \upsilon y : z+ \xi^{-1} \zeta_2 xy : t + \xi^{-1}\tau xy)
\end{align*}
for the latter case. Thus it follows that:
\begin{equation*}
    \mathrm{Aut}(X)\cong \Big{(} {\mathbbm G}_a^2 \rtimes ( {\mathbbm G}_m^2 \times {\mathbbm Z}_6 ) \Big{)} \rtimes {\mathbbm Z}_2 \cong 
    ({\mathbbm G}_a^2 \rtimes  {\mathbbm G}_m^2) \rtimes ({\mathbbm Z}_2^2 \times {\mathbbm Z}_3), 
\end{equation*}
where the last isomorphism is because of the fact that ${\mathbbm Z}_2$ part, which is generated by the involution:
\[
(x:y:z:t) \mapsto (x:y:t:z), 
\]
acts trivially on the ${\mathbbm Z}_6$ part by construction.

\subsubsection{{\rm No. 121}}
Let $X$ be in the family No. 121. By a suitable coordinate change, we may assume that $X$ is defined as follows:
\begin{equation*}
    X_8: zw + t^2 + y^4 + ax^4y^2 + bx^6y + cx^8 = 0\subset \mathbbm{P}(1_x,2_y,3_z,4_t,5_w),
\end{equation*}
where $a$, $b$ and $c$ are constants. By applying (\ref{Aut}) we only need to consider the automorphism group $\mathrm{Aut}({\mathbbm P}(1_x,2_y,3_z,4_t); E, h)$, where $E$ is the divisor defined by $z=0$ and $h$ is the curve defined by $z=t^2 + y^4 + ax^4y^2 + bx^6y + cx^8 = 0$. Every automorphism $\phi\in\mathrm{Aut}({\mathbbm P}(1_x,2_y,3_z,4_t); E, h)$ is of the form
\begin{equation*}
    \phi(x:y:z:t) = (\xi x : \upsilon y : \zeta z : \tau_1 t + \tau_2 xz),
\end{equation*}
where $\xi$, $\upsilon$, $\zeta$ and $\tau_i$ are constants such that $\tau_1^2=1$. The automorphism $\phi$ can be decomposed into three types of automorphisms, that is, $\phi = \varphi \circ \psi_{x,y} \circ \psi_t$, where
\begin{align*}
\psi_t(x:y:z:t) &= (x : y : z : \tau_1 t), \\
    \psi_{x,y}(x:y:z:t) &= (\xi x : \upsilon y: z : t),\\
    \varphi(x:y:z:t) &= (x : y : \zeta z : t + \xi^{-1} \tau_2 xz). 
\end{align*}
We will determine the group ${\mathbbm W}={\rm Aut}^0(X)/{\rm Aut}(X)$ by case-by-case argument. We first consider the case where $b=0$. By the quasi-smooth condition the constant $c$ must be nonzero. Then this case can be divided into two cases either $a=0$ or $a\neq 0$. If $a=0$ then the subgroup of the automorphism group $\mathrm{Aut}({\mathbbm P}(1_x,2_y,3_z,4_t); E, h)$ corresponding to $\psi_{x,y}$ can be decomposed into the two automorphisms $\psi_x(x:y:z:t) = (\xi x : y : z : t)$ and $\psi_y(x:y:z:t) = (x : \upsilon y: z : t)$. Since the subgroup corresponding to $\psi_t$ is isomorphic to $\mathbbm{Z}_2$, we have $\mathbbm{W}\cong  \mathbbm{Z}_8 \times \mathbbm{Z}_4 \times \mathbbm{Z}_2$. If $a\neq 0$ then the pair $(\xi, \upsilon)$ corresponding to $\psi_{x,y}$ is generated by $(\eta, \eta^2)$ and $(1, \eta^4)$ where $\eta$ is a primitive eighth root of unity. Thus 
we have $\mathbbm{W}\cong  \mathbbm{Z}_8 \times \mathbbm{Z}_2 \times \mathbbm{Z}_2$. Next, we consider the case where $b\neq 0$ and $c\neq 0$. We have $\xi^8 = 1$. Since $b\neq 0$, $\upsilon$ is uniquely determined. Thus we have $\mathbbm{W}\cong \mathbbm{Z}_8 \times \mathbbm{Z}_2$. Finally we consider the case where $b\neq 0$ and $c=0$. If $a=0$ then there is the pair $(\xi, \upsilon) = (\zeta_4, \zeta_{24}^7)$ such that the order of the corresponding automorphism $\psi_{x,y}$ is equal to $24$, where $\zeta_4$ (resp. $\zeta_{24}$) is a primitive fourth (resp. twenty fourth) root of unity. Thus we have $\mathbbm{W}\cong \mathbbm{Z}_{24}$. If $a\neq 0$ then it has same automorphism subgroup as the case $b\neq 0$ and $c\neq 0$. Therefore, we have
\begin{equation*}
    \mathrm{Aut}(X)\cong (\mathbbm{G}_a\rtimes \mathbbm{G}_m)\rtimes
    \begin{dcases}
        \mathbbm{Z}_8 \times \mathbbm{Z}_2 \times \mathbbm{Z}_2 & \text{if $a\neq 0$ and $b=0$},\\
        \mathbbm{Z}_8 \times \mathbbm{Z}_4 \times \mathbbm{Z}_2 & \text{if $a=b=0$},\\
        \mathbbm{Z}_{8} \times {\mathbbm Z}_3 \times \mathbbm{Z}_2 & \text{if $b\neq 0$ and $c=0$},\\
        \mathbbm{Z}_8 \times \mathbbm{Z}_2 & \text{otherwise}.
    \end{dcases}
\end{equation*}

\subsubsection{{\rm No. 123 }}
Let $X$ be in the family No. 123. By a suitable coordinate change, we may assume that $X$ is defined as follows:
\begin{equation*}
    X_6: xw + zt + y^3 = 0\subset \mathbbm{P}(1_x,2_y,3_z,3_t,5_w).
\end{equation*}
By applying (\ref{Aut}), we have ${\rm Aut}(X) \cong \mathrm{Aut}({\mathbbm P}(1_x,2_y,3_z,3_t); E, h)$, where $E$ is the divisor defined by $x=0$ and $h$ is the curve defined by $x=zt + y^3 = 0$. An automorphism $\phi\in\mathrm{Aut}({\mathbbm P}(1_x,2_y,3_z,3_t); E, h)$ is of the form: 
\begin{equation*}
    \phi(x:y:z:t) = (\xi x : \upsilon_1 y + \upsilon_2 x^2 : \zeta_1 z + \zeta_2 t + \zeta_3 xy + \zeta^4 x^3 : \tau_1 z + \tau_2 t + \tau_3 xy + \tau_4 x^3),
\end{equation*}
where $\xi$, $\upsilon_i$, $\zeta_i$ and $\tau_i$ are constants satisfying $\upsilon_1^3 = 1$. The automorphism $\phi$ can be decomposed into three types of automorphisms, that is, $\phi = \varphi \circ \psi_{z,t} \circ \psi_y$, where
\begin{align*}
\psi_y(x:y:z:t) &= (x : \upsilon_1 y : z : t), \\
    \psi_{z,t}(x:y:z:t) &= (x : y:  \zeta_1 z + \zeta_2 t : \tau_1 z + \tau_2 t),\\
    \varphi(x:y:z:t) &= (\xi x : y + \upsilon_2 x^2: z + {\color{black}{\upsilon_1^{-1}}} \zeta_3 xy + \zeta_4 x^3 : t + {\color{black}{\upsilon_1^{-1}}} \tau_3 xy + \tau_4 x^3). 
\end{align*}
The sub-group of $\mathrm{Aut}({\mathbbm P}(1_x,2_y,3_z,3_t); E, h)$ composed of automorphisms of type $\psi_y$ is isomorphic to ${\mathbbm Z}_3$, meanwhile that composed of automorphisms of type $\varphi$ is isomorphic to ${\mathbbm G}_a^5 \times {\mathbbm G}_m$. On the other hand, since the restriction of $\psi_{z,t}$ onto $E=\{ x=0 \} \cong {\mathbbm P} (2_y, 3_t, 3_w)$: 
\[
\psi_{z,t}|_E: (y:z:t) \mapsto \big{(} y: \zeta_1 z + \zeta_2 t: \tau_1 z + \tau_2 t \big{)} 
\]
must leave the curve $h$ defined by $zt+y^3=0$ invariant, the quadruple $(\zeta_1, \zeta_2, \tau_1, \tau_2)$ is described as follows:
\[
(\zeta_1, \zeta_2, \tau_1, \tau_2) = ( \zeta_1, 0, 0, \zeta_1^{-1}) \quad {\rm or} \quad (0, \zeta_2, \zeta_2^{-1}, 0),
\]
where $\zeta_1, \zeta_2 \in k^\ast$. Thus the sub-group generated by automorphisms of type $\psi_{z,t}$ is isomorphic to ${\mathbbm G}_m \rtimes {\mathbbm Z}_2$, where ${\mathbbm Z}_2$ acts on ${\mathbbm G}_m$ by means of the inverse action. As a consequence, we have finally:
\begin{equation*}
    \mathrm{Aut}(X)\cong (\mathbbm{G}_a^5\rtimes \mathbbm{G}_m^2)\rtimes (\mathbbm{Z}_2\times \mathbbm{Z}_3).
\end{equation*}


\subsubsection{\color{black}{\rm No. 124}}
Let $X$ be in the family No. 124. By a suitable coordinate change, we may assume that $X$ is defined as follows:
\begin{equation*}
    X_{10}: zw + t^2 + y^5 + a_1x^4y^3 + a_2x^6y^2 + a_3x^8y + a_4x^{10} \subset \mathbbm{P}(1_x,2_y,3_z,5_t,7_w),
\end{equation*}
where $a_i$ are constants. {\color{black}Since $X$ is quasi-smooth, at least one of the constants $a_3$ or $a_4$ must be nonzero.} To determine the automorphism group of $X$, it suffices to consider the automorphism group $\mathrm{Aut}({\mathbbm P}(1_x,2_y,3_z,5_t); E, h)$ by applying (\ref{Aut}), where $E$ is the divisor defined by $z=0$ and $h$ is the curve defined by $z=t^2 + y^5 + a_1x^4y^3 + a_2x^6y^2 + a_3x^8y + a_4x^{10} = 0$. Every automorphism $\phi\in\mathrm{Aut}({\mathbbm P}(1_x,2_y,3_z,5_t); E, h)$ is of the form
\begin{equation*}
    \phi(x:y:z:t) = (\xi x : \upsilon y : \zeta z : \tau_1 t + \tau_2 yz + \tau_3 x^2z),
\end{equation*}
where $\xi$, $\upsilon$, $\zeta$ and $\tau_i$ are constants satisfying $\xi^{10} = 1$, $\upsilon^5 = 1$ and $\tau_1^2 = 1$. The automorphism $\phi$ can be decomposed into three types of automorphisms, that is, $\phi = \varphi_3 \circ \varphi_2 \circ \varphi_1$, where
\begin{align*}
    \varphi_1(x:y:z:t) &= (x : y : z : \tau_1 t),\\
    \varphi_2(x:y:z:t) &= (\xi x : \upsilon y: z : t),\\
    \varphi_3(x:y:z:t) &= (x : y : \zeta z : t + {\color{black}{\upsilon^{-1}}} \tau_2 yz + {\color{black}{\xi^{-2}}} \tau_3 x^2z).
\end{align*}
The sub-group of $\mathrm{Aut}({\mathbbm P}(1_x,2_y,3_z,5_t); E, h)$ generated by automorphisms of $\varphi_1$ (resp. $\varphi_3$) is isomorphic to ${\mathbbm Z}_2$ (resp. ${\mathbbm G}_a^2 \rtimes {\mathbbm G}_m$). {\color{black}In what follows, for the time being we observe the sub-group $G$ generated by automorphisms of the form $\varphi_2$. The argument can be divided into two cases: either $a_4\neq 0$ or $a_4=0$. 

At first, we will deal with the case where $a_4\neq 0$. If $a_1=a_2=a_3=0$, then $\xi = \zeta_{10}^j$, $\upsilon = \zeta_5^i$ $(1\leqq i \leqq 5, 1\leqq j \leqq 10)$, where $\zeta_5$ and $\zeta_{10}$ are respectively a primitive fifth root of unity, a primitive tenth root of unity. Thence we have $G\cong {\mathbb Z}_5 \times {\mathbb Z}_{10}\cong {\mathbbm Z}_2 \times {\mathbbm Z}_5^2$. Otherwise, the direct computation says that $G$ is isomorphic to ${\mathbbm Z}_{10} \cong {\mathbbm Z}_2 \times {\mathbbm Z}_5$. 

Next, we consider the case where $a_4=0$, that is, $a_3\neq 0$. If $a_1=a_2=0$, then we have $\xi^8 \upsilon =1$ with $\upsilon = \zeta_5^i$ $(1 \leqq i \leqq 5)$, hence $G$ is isomorphic to ${\mathbbm Z}_{40}\cong {\mathbbm Z}_5 \times {\mathbbm Z}_8$. On the other hand, if $a_2 \neq 0$, then we have $\xi^2=\upsilon$ with $\upsilon = \zeta_5^i$ $(1 \leqq i \leqq 5)$. Hence $G \cong {\mathbbm Z}_2 \times {\mathbbm Z}_5$. Finally if $a_1\neq 0$ and $a_2=0$, then we have $\xi^4=\upsilon^2$ with $\upsilon = \zeta_5^i$ $(1 \leqq i \leqq 5)$, which says that $G \cong {\mathbbm Z}_4 \times {\mathbbm Z}_5$. 
Summarizing, we have finally: 
\begin{equation*}
    \mathrm{Aut}(X)\cong
    \begin{dcases}
        \big{(} \mathbbm{G}_a^2\rtimes \mathbbm{G}_m \big{)} \rtimes \big{(} \mathbbm{Z}_{2}^2\times \mathbbm{Z}_5^2 \big{)} & \text{if $a_1 = a_2 = a_3=0$}\\
        \big{(} \mathbbm{G}_a^2\rtimes \mathbbm{G}_m \big{)} \rtimes \big{(} {\mathbbm Z}_2 \times {\mathbbm Z}_5 \times {\mathbbm Z}_8 \big{)} & \text{if $a_1 = a_2 = a_4=0$}\\
        \big{(} \mathbbm{G}_a^2\rtimes \mathbbm{G}_m \big{)} \rtimes \big{(} \mathbbm{Z}_{2} \times {\mathbbm Z}_4 \times {\mathbbm Z}_5 \big{)} & \text{if $a_1 \neq 0$ and $a_2=a_4=0$}\\
        \big{(} \mathbbm{G}_a^2\rtimes \mathbbm{G}_m \big{)} \rtimes \big{(} \mathbbm{Z}_{2}^2 \times {\mathbbm Z}_5 \big{)}  & \text{otherwise.}
    \end{dcases}
\end{equation*}
}


\subsubsection{{\rm No. 125}}
Let $X$ be in the family No. 125. By a suitable coordinate change, we may assume that $X$ is defined as follows:
\begin{equation*}
    X_{12}: tw + z^3 + (a_1y^2 + a_2x^3y + a_3x^6)x^2z + y^4 + f_3(x,y)x^6y + bx^{12}  = 0\subset \mathbbm{P}(1_x,3_y,4_z,5_t,7_w),
\end{equation*}
where $a_i$ and $b$ are constants and $f_3(x,y)$ are quasi-homogeneous polynomials of degree $3$. By applying (\ref{Aut}) we only need to consider the automorphism group $\mathrm{Aut}({\mathbbm P}(1_x,3_y,4_z,5_t); E, h)$, where $E$ is the divisor defined by $t=0$ and $h$ is the curve defined by $t=z^3 + x^2(a_1y^2 + a_2x^3y + a_3x^6)z + y^4 + b_1x^6y^2 + b_2x^9y + b_3x^{12} = 0$. Every automorphism $\phi\in\mathrm{Aut}({\mathbbm P}(1_x,3_y,4_z,5_t); E, h)$ should be of the form
\begin{equation*}
    \phi(x:y:z:t) = (\xi x : \upsilon y : \zeta z : \tau t),
\end{equation*}
where $\xi$, $\upsilon$, $\zeta$ and $\tau$ are constants satisfying $\xi^{12}=1$, $\upsilon^4=1$ and $\zeta^3=1$. We have
\begin{equation*}
    \mathrm{Aut}(X)\cong \mathbbm{G}_m \times
    \begin{dcases}
        \mathbbm{Z}_{3} \times {\mathbbm Z}_4 & \text{if at least one $a_i$ is nonzero and $f_3(x,y)\neq 0$}, \\
        \mathbbm{Z}_{3}\times \mathbbm{Z}_4^2 & \text{if $a_1=a_2=0$, $a_3\neq 0$ and $f_3(x,y)\equiv 0$},\\
        \mathbbm{Z}_{3}^2 \times \mathbbm{Z}_4 & \text{if $a_1=a_2=a_3=0$ and $f_3(x,y)\neq 0$},\\
        \mathbbm{Z}_{3}^2 \times \mathbbm{Z}_{4}^2 & \text{if $a_1=a_2=a_3= 0$ and $f_3(x,y)\equiv 0$}.
    \end{dcases}
\end{equation*}

\subsubsection{{\rm No. 126}}
Let $X$ be in the family No. 126. By a suitable coordinate change, we may assume that $X$ is defined as follows:
\begin{equation*}
    X_6: xw + yt + z^2 = 0\subset \mathbbm{P}(1_x,2_y,3_z,4_t,5_w).
\end{equation*}
By applying (\ref{Aut}), we only need to consider the automorphism group $\mathrm{Aut}({\mathbbm P}(1_x,2_y,3_z,4_t); E, h)$, where $E$ is the divisor defined by $x=0$ and $h$ is the curve defined by $x=yt + z^2 = 0$. Every automorphism $\phi\in\mathrm{Aut}({\mathbbm P}(1_x,2_y,3_z,4_t); E, h)$ is of the form
\begin{equation*}
    \phi(x:y:z:t) = (\xi x : \upsilon_1 y + \upsilon_2 x^2: \zeta_1 z + \zeta_2 xy + \zeta_3 x^3: \tau_1 t + \tau_2 xz + \tau_3 x^2y + \tau_4 x^4),
\end{equation*}
where $\xi$, $\upsilon_i$, $\zeta_i$ and $\tau_i$ are constants. The automorphism $\phi$ can be decomposed into three types of automorphisms, that is, $\phi = \varphi \circ \psi_{y,t} \circ \psi_z$, where
\begin{align*}
 \psi_z(x:y:z:t) &= (x : y : \zeta_1 z : t), \\ 
    \psi_{y,t}(x:y:z:t) &= (x : \upsilon_1 y: z : \tau_1 t),\\
   \varphi(x:y:z:t) &= (\xi x : y + \upsilon_2 x^2 : z + \upsilon_1^{-1} \zeta_2 xy + \zeta_3 x^3: t + \zeta_1^{-1} \tau_2 xz + \upsilon_1^{-1} \tau_3 x^2y + \tau_4 x^4). 
\end{align*}
The equations $\upsilon_1 \tau_1 = 1$ and $\zeta_1^2 = 1$ hold. These imply that the subgroups of $\mathrm{Aut}({\mathbbm P}(1_x,2_y,3_z,4_t); E, h)$ corresponding to $\psi_{y,t}$ and $\psi_z$ are isomorphic to $\mathbbm{G}_m$ and $\mathbbm{Z}_2$, respectively. Therefore, we have
\begin{equation*}
    \mathrm{Aut}(X)\cong (\mathbbm{G}_a^6 \rtimes \mathbbm{G}_m^2) \rtimes \mathbbm{Z}_2.
\end{equation*}

\subsubsection{{\rm No. 127}}
Let $X$ be in the family No. 127. By a suitable coordinate change, we may assume that $X$ is defined as follows:
\begin{equation*}
    X_{12}: tw + y^4 + z^3 + ax^4z + bx^{6} = 0\subset \mathbbm{P}(2_x,3_y,4_z,5_t,7_w),
\end{equation*}
where $a$ and $b$ are constants. Note that the case $(a,b)=(0,0)$ does not occur since $X_{12}$ is assumed to be quasi-smooth. In consideration of (\ref{Aut}), we have only to consider the automorphism group $\mathrm{Aut}({\mathbbm P}(2_x,3_y,4_z,5_t); E, h)$, where $E$ is the divisor defined by $t=0$ and $h$ is the curve defined by $t=y^4 + z^3 + ax^4z + bx^{6} = 0$. Every automorphism $\phi\in\mathrm{Aut}({\mathbbm P}(2_x,3_y,4_z,5_t); E, h)$ is of the form
\begin{equation*}
    \phi(x:y:z:t) = (\xi x : \upsilon y : \zeta z : \tau t),
\end{equation*}
where $\xi$, $\upsilon$, $\zeta$ and $\tau$ are non-zero constants. The automorphism $\phi$ can be decomposed into three types of automorphisms, that is, $\phi = \varphi \circ \psi_{x,z} \circ \psi_y$, where
\begin{align*}
\psi_y(x:y:z:t) &= (x : \upsilon y : z : t), \\
    \psi_{x,z}(x:y:z:t) &= (\xi x : y: \zeta z : t ),\\
    \varphi(x:y:z:t) &= (x : y : z : \tau t). 
\end{align*}
The sub-groups generated by automorphisms corresponding to $\psi_y$ (resp. $\varphi$) is isomorphic to ${\mathbbm Z}_4$ (resp. ${\mathbbm G}_m$). The restriction $\phi|_E$ of $\phi$ onto $E=\{ t=0 \} \cong {\mathbbm P}(2_x,3_y,4_z)$:
\[
\phi|_E: (x:y:z) \mapsto \big{(} \xi x : \upsilon y: \zeta z \big{)}
\]
should leave the curve defined by $y^4+z^3+ax^4z+bx^6=0$ invariant, in other words, the following equalities have to be satisfied:
\[
\upsilon^4=1, \, \zeta^3=1, \, a\xi^4 \zeta = a, \, b\xi^6 =b.
\]
Let $G$ be the sub-group of ${\rm Aut} ({\mathbbm P}(2_x,3_y,4_z,5_t); E, h)$ generated by automorphisms $\psi_{x,z}$. The straightforward computation says that:
\begin{equation*}
    G \cong
    \begin{dcases}
       {\mathbbm Z}_2 \times {\mathbbm Z}_3 &\text{if $a\neq 0$ and $b\neq 0$},\\
        {\mathbbm Z}_3 \times {\mathbbm Z}_4 &\text{if $a\neq 0$ and $b= 0$},\\
        {\mathbbm Z}_2 \times {\mathbbm Z}_3^2 &\text{if $a = 0$ and $b\neq 0$}.
    \end{dcases}
\end{equation*}
Therefore we have
\begin{equation*}
    \mathrm{Aut}(X)\cong
    \begin{dcases}
        \mathbbm{G}_m \times ({\mathbbm Z}_2 \times {\mathbbm Z}_3 \times {\mathbbm Z}_4)&\text{if $a\neq 0$ and $b\neq 0$},\\
        \mathbbm{G}_m \times ({\mathbbm Z}_3 \times {\mathbbm Z}_4^2)&\text{if $a\neq 0$ and $b= 0$},\\
        \mathbbm{G}_m \times ({\mathbbm Z}_2 \times {\mathbbm Z}_3^2 \times {\mathbbm Z}_4)&\text{if $a = 0$ and $b\neq 0$}.
    \end{dcases}
\end{equation*}

\subsubsection{{\rm No. 128}}
Let $X$ be in the family No. 128. By a suitable coordinate change, we may assume that $X$ is defined as follows:
\begin{equation*}
    X_{12}: zw + t^2 + y^3 + ax^8y + bx^{12} = 0\subset \mathbbm{P}(1_x,4_y,5_z,6_t,7_w),
\end{equation*}
where $a$ and $b$ are constants. By applying (\ref{Aut}), we only need to consider the automorphism group $\mathrm{Aut}({\mathbbm P}(1_x,4_y,5_z,6_t); E, h)$, where $E$ is the divisor defined by $z=0$ and $h$ is the curve defined by $z=t^2 + y^3 + ax^8y + bx^{12} = 0$. Every automorphism $\phi\in\mathrm{Aut}({\mathbbm P}(1_x,4_y,5_z,6_t); E, h)$ is of the form
\begin{equation*}
    \phi(x:y:z:t) = (\xi x : \upsilon y : \zeta z : \tau_1 t + \tau_2 xz),
\end{equation*}
where $\xi$, $\upsilon$, $\zeta$ and $\tau_i$ are constants. The automorphism $\phi$ can be decomposed into three types of automorphisms, that is, $\phi = \varphi \circ \psi_{x,y} \circ \psi_t$, where
\begin{align*}
\psi_t(x:y:z:t) &= (x : y : z : \tau_1 t), \\ 
    \psi_{x,y}(x:y:z:t) &= (\xi x : \upsilon y: z : t ),\\
     \varphi(x:y:z:t) &= (x : y : \zeta z  : t + \xi^{-1} \tau_2 xz). 
\end{align*}
We consider the subgroup $G$ of $\mathrm{Aut}({\mathbbm P}(1_x,4_y,5_z,6_t); E, h)$ corresponding to $\psi_{x,y}$. Note that the restriction $\phi|_E$ of $\phi$ to $E=\{ z=0 \} \cong {\mathbbm P}(1_x, 4_y, 6_t)$ has to leave the curve defined by $t^2+y^3+ax^8y+bx^{12}=0$ invariant, which means that the following equalities must hold:
\[
\tau_1^2=1, \, \upsilon^3=1, \, a\xi^8\upsilon =a, \, b \xi^{12}= b.
\]
The straight forward computation allows to conclude:
\begin{equation*}
    G \cong
    \begin{dcases}
       {\mathbbm Z}_3 \times {\mathbbm Z}_4 &\text{if $a\neq 0$ and $b\neq 0$},\\
        {\mathbbm Z}_3 \times {\mathbbm Z}_8 &\text{if $a\neq 0$ and $b= 0$},\\
        {\mathbbm Z}_3^2 \times {\mathbbm Z}_4 &\text{if $a = 0$ and $b\neq 0$}.
    \end{dcases}
\end{equation*}
As a consequence, we have: 
\begin{equation*}
    \mathrm{Aut}(X)\cong
    \begin{dcases}
        (\mathbbm{G}_a \rtimes \mathbbm{G}_m) \rtimes (\mathbbm{Z}_2 \times {\mathbbm Z}_3 \times {\mathbbm Z}_4)&\text{if $a\neq 0$ and $b\neq 0$},\\
        (\mathbbm{G}_a \rtimes \mathbbm{G}_m) \rtimes (\mathbbm{Z}_2 \times {\mathbbm Z}_3 \times {\mathbbm Z}_8)&\text{if $a\neq 0$ and $b= 0$},\\
        (\mathbbm{G}_a \rtimes \mathbbm{G}_m) \rtimes ( {\mathbbm Z}_2 \times \mathbbm{Z}_3^2 \times \mathbbm{Z}_{4})&\text{if $a = 0$ and $b\neq 0$}.
    \end{dcases}
\end{equation*}

\subsubsection{{\rm No. 129}}
Let $X$ be in the family No. 129. By a suitable coordinate change, we may assume that $X$ is defined as follows:
\begin{equation*}
    X_{10}: yw + t^2 + xz^2 + x^5 = 0\subset \mathbbm{P}(2_x,3_y,4_z,5_t,7_w).
\end{equation*}
By applying (\ref{Aut}), we only need to consider the automorphism group $\mathrm{Aut}({\mathbbm P}(2_x,3_y,4_z,5_t); E, h)$, where $E$ is the divisor defined by $y=0$ and $h$ is the curve defined by $y=t^2 + xz^2 + x^5 = 0$. Every automorphism $\phi\in\mathrm{Aut}({\mathbbm P}(2_x,3_y,4_z,5_t); E, h)$ is of the form: 
\begin{equation*}
    \phi(x:y:z:t) = (\xi x : \upsilon y : \zeta z : \tau_1 t + \tau_2 xy),
\end{equation*}
where $\xi$, $\upsilon$, $\zeta$ and $\tau_i$ are constants. Note that the automorphism $\phi$ can be decomposed into three types of automorphisms, that is, $\phi = \varphi \circ \psi_{x,z} \circ \psi_t$, where
\begin{align*}
 \psi_t(x:y:z:t) &= (x : y : z : \tau_1 t), \\ 
    \psi_{x,z}(x:y:z:t) &= (\xi x : y: \zeta z : t ),\\
   \varphi(x:y:z:t) &= (x : \upsilon y : z  : t + \xi^{-1} \tau_2 xy). 
\end{align*}
The restriction of $\phi$ onto $E=\{ y=0 \} \cong {\mathbbm P}(2_x, 4_z, 5_t)$ yields:
\[
\phi|_E: (x:z:t) \mapsto \big{(} \xi x: \zeta z: \tau_1 t \big{)} 
\]
needs to leave the curve defined by $t^2+xz^2+x^5=0$ invariant, which means:
\[
\tau_1^2=1, \, \xi^5=1, \, \xi \zeta^2=1. 
\]
This implies in turn that the sub-group generated by $\psi_{x,z}$ and $\psi_t$ is isomorphic to ${\mathbbm Z}_2^2 \times {\mathbbm Z}_5$. 
Therefore we have: 
\begin{equation*}
    \mathrm{Aut}(X)\cong (\mathbbm{G}_a\rtimes \mathbbm{G}_m) \rtimes (\mathbbm{Z}_2^2 \times {\mathbbm Z}_5). 
\end{equation*}

\subsubsection{{\rm No. 130}}
Let $X$ be in the family No. 130. By a suitable coordinate change, we may assume that $X$ is defined as follows:
\begin{equation*}
    X_{12}: zw + t^2 + y^3 + x^4 = 0  \subseteq {\mathbbm P}(3_x,4_y,5_z,6_t,7_w).
\end{equation*}
By applying (\ref{Aut}), we only need to consider the automorphism group $\mathrm{Aut}({\mathbbm P}(3_x,4_y,5_z,6_t); E, h)$, where $E$ is the divisor defined by $z=0$ and $h$ is the curve defined by $z=t^2 + y^3 + x^4 = 0$. Every automorphism $\phi\in\mathrm{Aut}({\mathbbm P}(3_x,4_y,5_z,6_t); E, h)$ is of the form: 
\begin{equation*}
    \phi(x:y:z:t) = (\xi x : \upsilon y : \zeta z : \tau t),
\end{equation*}
where $\xi$, $\upsilon$, $\zeta$ and $\tau$ are constants such that $\xi^4= 1$, $\upsilon^3 = 1$, $\zeta \neq 0$ and $\tau^2 = 1$. The automorphism $\phi$ can be decomposed into four types of automorphisms, that is, $\phi = \psi_t \circ \psi_z \circ \psi_y \circ \psi_x$, where
\begin{align*}
    \psi_x(x:y:z:t) &= (\xi x : y: z : t ),\\
    \psi_y(x:y:z:t) &= (x : \upsilon y: z : t ),\\
    \psi_z(x:y:z:t) &= (x : y : \zeta z  : t),\\
    \psi_t(x:y:z:t) &= (x : y : z : \tau t).
\end{align*}
Therefore we have: 
\begin{equation*}
    \mathrm{Aut}(X) \cong \mathbbm{G}_m\times \mathbbm{Z}_2\times \mathbbm{Z}_3\times \mathbbm{Z}_4.
\end{equation*}

\section{Proof of Theorem \ref{thm:form}}\label{section4}
\subsection{}\label{4-1}
In this section, we will prove Theorem \ref{thm:form}. Let $\Bbbk$ be a field of characteristic zero and let us denote by $k=\overline{\Bbbk}$ the algebraic closure of $\Bbbk$ in a fixed universal domain. Let $Y={\rm Proj}_{\Bbbk}(S)$ be a projective variety defined over $\Bbbk$ such that its base extension $X:=Y_k$ is a non-smooth weighted Fano threefold hypersurface containing the affine $3$-space ${\mathbbm A}_k^3$, namely, belonging to the families No.105, 111, 113, 118, 119, 123 and 126. In what follows, we will show that $Y$ contains always the affine $3$-space ${\mathbbm A}_{\Bbbk}^3$ defined over $\Bbbk$. 
\subsection{Strategy of Proof}\label{4-2}
Before dealing with $\Bbbk$-forms of Fanos in question separately, we will explain a common strategy: The defining equation of the base extension $X=Y_k$ in the corresponding weighted projective space is described as follows by a suitable change of coordinates:
\[
xw+ f(y,z,t) =0 \, \subseteq \, {\mathbb P}\big{(} 1_x, a_y, b_z, c_t, d_w \big{)}, 
\]
where $f(y,z,t)$ is a suitable quasi-homogeneous polynomial of degree $d+1$ in $y,z,t$. In our case, note that the point ${\mathsf p}_w=[0:0:0:0:1]$ is defined over $\Bbbk$ (because ${\mathsf p}_w$ is a unique singular point of $X$ of the highest index $d$), furthermore, the special surface $H=\{ x=0 \} \cap X$ is also defined over $\Bbbk$ (because $H$ is a unique member of $|{\mathscr O}_X(1)|$ having ${\mathsf p}_w$ as a cone singularity). In particular, the maximal ideal ${\mathfrak m}_{{\mathsf p}_w}$ can be regarded as a homogeneous ideal of $S$. Then we have an injective homomorphism of graded $\Bbbk$-algebras:
\[
{\rm Sym}^{\bullet} \Big{(} {\mathfrak m}_{{\mathsf p}_w} \big{/} {\mathfrak m}_{{\mathsf p}_w}^2 \Big{)} \hookrightarrow S, 
\]
which gives rise to a rational map defined over $\Bbbk$:
\[
\pi: \, Y={\rm Proj}_{\Bbbk} \big{(} S \big{)} \, \dasharrow \, 
{\rm Proj}_{\Bbbk} \Big{(} {\rm Sym}^{\bullet} \big{(} {\mathfrak m}_{{\mathsf p}_w} \big{/} {\mathfrak m}_{{\mathsf p}_w}^2 \big{)} \Big{)}=:P
\]
By construction, the base extension $\pi_k$ of $\pi$ coincides with the projection of $X$ from ${\mathsf p}_w$:
\[
\pi_{{\mathsf p}_w}: X \, \dasharrow \, {\mathbbm P} \big{(} 1_x, a_y, b_z, c_t \big{)}, \quad \big{[} x:y:z:t:w \big{]} \mapsto \big{[} x:y:z:t \big{]}.
\]
As seen in \ref{3-1}, $\pi_k=\pi_{{\mathsf p}_w}$ yields an isomorphism between the open subsets $X \backslash H$ and ${\mathbbm P}(1_x,a_y,b_z, c_t) \backslash E$, where $E$ is the divisor defined by $x=0$. In particular, $E$ is also defined over $\Bbbk$, and $\pi$ yields an isomorphism $Y \backslash H \cong P \backslash E$. 
Note that the base extension of $P\backslash E$ to the algebraic closure $k$ is isomorphic to ${\mathbbm A}_{k}^3$, however it does not necessarily imply that $P\backslash E$ itself is isomorphic to ${\mathbbm A}_{\Bbbk}^3$ because we do not know so far the triviality of forms of the affine spaces of dimension greater than or equal to three. We will proceed depending on whether $b_z=c_t$ or not. In case that $b_z=c_t$, the coordinate stratum curve $\ell_{zt}=\{ x=y=0 \}$ is actually defined over $\Bbbk$, and we will find an appropriate mobile linear pencil ${\mathscr L}$ defined over $\Bbbk$, hence ${\mathscr L}$ can be regarded as a linear pencil on $P$ with $\ell_{zt}$ being its base locus such that the restriction of the rational map:
\[
\varphi_{{\mathscr L}}: \, P \, \dasharrow \, {\mathbbm P} \big{(} {\mathscr L}^{\vee} \big{)} \cong {\mathbbm P}_{\Bbbk}^1
\]
onto the open subset $P \backslash E$ gives rise to a trivial ${\mathbbm A}^2$-bundle over the affine line ${\mathbbm A}_{\Bbbk}^1$ to conclude $Y \backslash H \cong P \backslash E \cong {\mathbbm A}_{\Bbbk}^1 \times {\mathbbm A}_{\Bbbk}^2 \cong {\mathbbm A}_{\Bbbk}^3$. Meanwhile, in case of $b_z<c_t$, the point ${\mathsf q}_t=[0:0:0:1]$ yields a $\Bbbk$-rational point of $P$, and we shall find a suitable linear system ${\mathscr H}$ of dimension two defined over $\Bbbk$ having ${\mathsf q}_t$ as a unique base point such that the restriction of the rational map:
\[
\varphi_{{\mathscr H}}: \, P \, \dasharrow \, {\mathbbm P} \big{(} {\mathscr H}^{\vee} \big{)} \cong {\mathbbm P}_{\Bbbk}^2
\]
to $P \backslash E$ yields a trivial ${\mathbbm A}^1$-bundle over the affine plane ${\mathbbm A}_{\Bbbk}^2$ to see that $P\backslash E \cong {\mathbbm A}_{\Bbbk}^2 \times {\mathbbm A}_{\Bbbk}^1 \cong {\mathbbm A}_{\Bbbk}^3$. Thus the main part for the proof lies in how to find the linear pencils ${\mathscr L}$ or the linear systems ${\mathscr H}$ with desired properties for each family in question except for No.105.  
\subsection{Proof}\label{4-3}
In what follows, we will explain how to find the desired linear pencil or the linear system of dimension two on $P$ for each family from No.111, 113, 118, 119, 123 and 126, separately. Note that for No.105 we will proceed in a more direct way: 
\subsubsection{{\rm No. 105}}\label{105}
In this case, $P$ is a Severi-Brauer threefold such that $E$ is a twisted linear space of codimension one. Hence $P$ is isomorphic to ${\mathbbm P}_{\Bbbk}^3$ with $E$ being a hyperplane (cf. \cite{GS}). In particular, $P\backslash E \cong {\mathbbm A}_{\Bbbk}^3$. 

\subsubsection{{\rm No. 111}}\label{111}
In this case, we have $P_k \cong {\mathbbm P}(1_x,1_y,1_z,2_t)$, hence the point ${\mathsf q}_t=[0:0:0:1]$ is defined over $\Bbbk$. Then the linear system of dimension two:
\[
{\mathscr H}:= \big{|} {\mathscr O}_P (E) \otimes {\mathfrak m}_{{\mathsf q}_t} \big{|}
\]
on $P$ satisfies the desired properties. 

\subsubsection{{\rm No. 113}}\label{113}
In this case, $P_k \cong {\mathbbm P} ( 1_x, 1_y, 2_z, 2_t)$ and the line $\ell_{zt}=\{ x=y=0 \}$ is defined over $\Bbbk$. Then we can check that the linear pencil:
\[
{\mathscr L}:= \big{|} {\mathscr O}_P (E) \otimes {\mathscr I}_{\ell_{zt}} \big{|}
\]
satisfies the desired properties. 

\subsubsection{{\rm No. 118}}\label{118}
In this case, we see $P_k \cong {\mathbbm P} ( 1_x, 1_y, 2_z, 3_t)$ and the line $\ell_{zt}=\{ x=y=0 \}$ is defined over $\Bbbk$. As in \ref{113}, the linear pencil:
\[
{\mathscr L}:= \big{|} {\mathscr O}_P (E) \otimes {\mathscr I}_{\ell_{zt}} \big{|}
\]
satisfies the desired properties. 

\subsubsection{{\rm No. 119}}\label{119}
In this case, we have $P_k \cong {\mathbbm P}(1_x,2_y,2_z,3_t)$, so that the point ${\mathsf q}_t=[0:0:0:1]$ is defined over $\Bbbk$. The linear system of dimension two:
\[
{\mathscr H}:= \big{|} {\mathscr O}_P (2E) \otimes {\mathfrak m}_{{\mathsf q}_t} \big{|}
\]
satisfies the desired properties. 

\subsubsection{{\rm No. 123}}\label{123}
In this case, we see $P_k \cong {\mathbbm P} ( 1_x, 2_y, 3_z, 3_t)$, hence $\ell_{zt}=\{ x=y=0 \}$ is defined over $\Bbbk$. Then the linear pencil on $P$:
\[
{\mathscr L}:= \big{|} {\mathscr O}_P (2E) \otimes {\mathscr I}_{\ell_{zt}} \big{|}
\]
satisfies the desired properties. 

\subsubsection{{\rm No. 126}}\label{126}
In this case, we see $P_k \cong {\mathbbm P} ( 1_x, 2_y, 3_z, 4_t)$, and $\ell_{zt}=\{ x=y=0 \}$ is defined over $\Bbbk$. Then the linear pencil on $P$:
\[
{\mathscr L}:= \big{|} {\mathscr O}_P (2E) \otimes {\mathscr I}_{\ell_{zt}} \big{|}
\]
satisfies the desired properties.

\bibliographystyle{amsplain}

\end{document}